\documentclass[11pt,reqno,tbtags]{amsart}
\pdfoutput=1
\usepackage[]{amsmath,amssymb,amsfonts,latexsym,amsthm,enumerate}
\usepackage{bbm}
\usepackage{graphicx}
\usepackage{enumerate}%,graphicx,subfigure}
\usepackage[numeric,initials,nobysame]{amsrefs}
\numberwithin{equation}{section}
\newtheorem{maintheorem}{Theorem}
\newtheorem{maincoro}[maintheorem]{Corollary}

\newtheorem{theorem}{Theorem}[section]
\newtheorem*{theorem*}{Theorem}
\newtheorem{lemma}[theorem]{Lemma}
\newtheorem{claim}[theorem]{Claim}

\theoremstyle{definition}{

\newtheorem{remark}[theorem]{Remark}
\newtheorem{definition}[theorem]{Definition}
\newtheorem*{definition*}{Definition}

}

\theoremstyle{remark}{
\newtheorem*{remark*}{Remark}

}

\newcommand{\Z}{\mathbb Z}
\newcommand{\E}{\mathbb{E}}
\renewcommand{\P}{\mathbb{P}}
\DeclareMathOperator{\var}{Var}

\newcommand{\gap}{\lambda}
\newcommand{\sob}{\alpha_{\text{\tt{s}}}}
\newcommand{\tmix}{t_\textsc{mix}}
\newcommand{\tv}{{\textsc{tv}}}
\newcommand{\Po}{\operatorname{Po}}

\newcommand{\one}{\mathbbm{1}}

\renewcommand{\epsilon}{\varepsilon}
\renewcommand{\phi}{\varphi}

\newcommand{\tX}{\tilde{X}}
\newcommand{\sparse}{\mathcal{S}}
\newcommand{\ltwo}{{\mathfrak{m}}}

\DeclareMathOperator{\dist}{dist}
\newcommand{\gapinf}{\lambda^\star}
\newcommand{\sobinf}{\alpha_{\text{\tt{s}}}^\star}

\date{}

\begin{document}
\title{Cutoff for the Ising model on the lattice}

\author{Eyal Lubetzky}
\address{Eyal Lubetzky\hfill\break
Microsoft Research\\
One Microsoft Way\\
Redmond, WA 98052-6399, USA.}
\email{eyal@microsoft.com}
\urladdr{}

\author{Allan Sly}
\address{Allan Sly\hfill\break
Microsoft Research\\
One Microsoft Way\\
Redmond, WA 98052-6399, USA.}
\email{allansly@microsoft.com}
\urladdr{}

\begin{abstract}
Introduced in 1963, Glauber dynamics is one of the most practiced and extensively studied methods for sampling the Ising model on lattices.
It is well known that at high temperatures, the time it takes this chain to mix in $L^1$ on a system of size $n$ is $O(\log n)$. Whether in this regime there is \emph{cutoff}, i.e.\ a sharp transition in the $L^1$-convergence to equilibrium, is a fundamental open problem: If so, as conjectured by Peres, it would imply that mixing occurs abruptly at $(c+o(1))\log n$
for some fixed $c>0$,
thus providing a rigorous stopping rule for this MCMC sampler. However, obtaining the precise asymptotics of the mixing and proving cutoff can be extremely challenging even for
fairly simple Markov chains.
Already for the one-dimensional Ising model, showing cutoff is a longstanding open problem.

We settle the above by establishing cutoff and its location at the high temperature regime of the
Ising model on the lattice
with periodic boundary conditions. Our results hold for any dimension and at any temperature
where there is strong spatial mixing: For $\Z^2$ this carries all the way to the critical temperature.
Specifically, for fixed $d\geq 1$, the continuous-time Glauber dynamics
for the Ising model on $(\Z/n\Z)^d$ with periodic boundary conditions has cutoff at
$(d/2\lambda_\infty)\log n$, where $\lambda_\infty$ is the spectral gap of the dynamics on the
infinite-volume lattice. To our knowledge, this is the first time where cutoff is shown for a Markov chain where
even understanding its stationary distribution is limited.

The proof hinges on a new technique for translating $L^1$-mixing to $L^2$-mixing of projections of the chain,
which enables the application of logarithmic-Sobolev inequalities. The technique is general and carries to other
monotone and anti-monotone spin-systems,
e.g.\ gas hard-core, Potts, anti-ferromagentic Ising, arbitrary boundary conditions, etc.
\end{abstract}

\maketitle

\vspace{-1cm}

\section{Introduction}\label{sec:intro}

The total-variation \emph{cutoff phenomenon} describes a sharp transition in the $L^1$-mixing of a finite ergodic Markov chain:
Over a negligible time period, the distance of the chain from equilibrium drops abruptly from near its maximum to near $0$.
Though believed to be widespread, including many important families of chains arising from statistical physics,
cutoff has been rigorously shown only in relatively few cases (ones where the stationary distribution is completely understood and
has many symmetries, e.g.\ uniform on the symmetric group).
Here we establish cutoff for Glauber dynamics for the Ising model, one of the most studied models in mathematical physics.

Already establishing the order of the mixing time is in many cases challenging, with an entire industry devoted to the study of such problems.
Proving cutoff and its location entails not only obtaining the order, but also deriving the precise asymptotics of the time it takes the chain to mix.
In his 1995 survey of the cutoff phenomenon Diaconis \cite{Diaconis} wrote
``\emph{At present writing, proof of a cutoff is a difficult, delicate
affair, requiring detailed knowledge of the chain, such as all
eigenvalues and eigenvectors. Most of the examples where this
can be pushed through arise from random walk on groups, with
the walk having a fair amount of symmetry}''.
To this date this essentially remained to be the case, and present technology
(e.g., representation theory, spectral theory, techniques for analyzing 1-dimensional chains, etc.) does not suffice for proving cutoff in high-dimensional
chains with limited understanding of their stationary distribution, such as Glauber dynamics for the Ising model (stochastic Ising model) on the 3-dimensional lattice.

Introduced in 1963 \cite{Glauber}, %heat-bath
Glauber dynamics for the Ising model on the lattice (see Section~\ref{sec:intro-previous} for formal definitions) is one of the most practiced methods to sample
the Gibbs distribution, and an extensively studied dynamical system in itself, having a rich interplay of properties with the static stationary distribution.
For instance, as we describe in Section~\ref{sec:intro-previous}, it is known that on $(\Z/\Z_n)^2$, at the critical inverse-temperature $\beta_c$ for uniqueness of the static Gibbs distribution,
the spectral gap of the Markov semigroup generator of the Glauber dynamics exhibits a phase-transition from being uniformly bounded
to tending to $0$ exponentially fast in $n$.
It is further known that at high temperatures on $(\Z/\Z_n)^d$ this dynamics mixes in time $O(\log n)$,
yet the precise asymptotics of the $L^1$-mixing time were unknown even in the one-dimensional case: It is an open problem of
Peres (cf.\ \cite{LPW}) to determine cutoff for the Ising model on $\Z/\Z_n$.

The only underlying geometry for which cutoff for the Ising model had so far been established is the complete graph (\cites{LLP,DLP}),
where the high symmetry reduces the analysis to a birth-and-death magnetization chain.
However, this sheds no light on the existence of cutoff for lattices, where there is no such reduction.
Peres (\cites{LLP,LPW}) conjectured that in any dimension $d$, Glauber dynamics for the Ising model on $(\Z/\Z_n)^d$ should exhibit cutoff.

Our main results confirm the above conjecture and moreover establish cutoff and its location in a wide range of spin system models and geometries.
We first formulate this for the classical two-dimensional Ising model.

\vspace{-0.25cm}

\begin{maintheorem}\label{mainthm-Z2}
Let $\beta_c=\frac12\log(1+\sqrt{2})$ be the critical inverse-temperature for the Ising model on $\Z^2$.
Then the continuous-time Glauber dynamics for the Ising model on $(\Z/n\Z)^2$ at inverse-temperature $0 \leq \beta < \beta_c$ with periodic
boundary conditions has cutoff at $\lambda_\infty^{-1}\log n $ with a window of $O(\log\log n)$, where
$\lambda_\infty$ is the spectral gap of the dynamics on the infinite-volume lattice.
\end{maintheorem}

\begin{figure}[h]
\centering \includegraphics[width=3.5in]{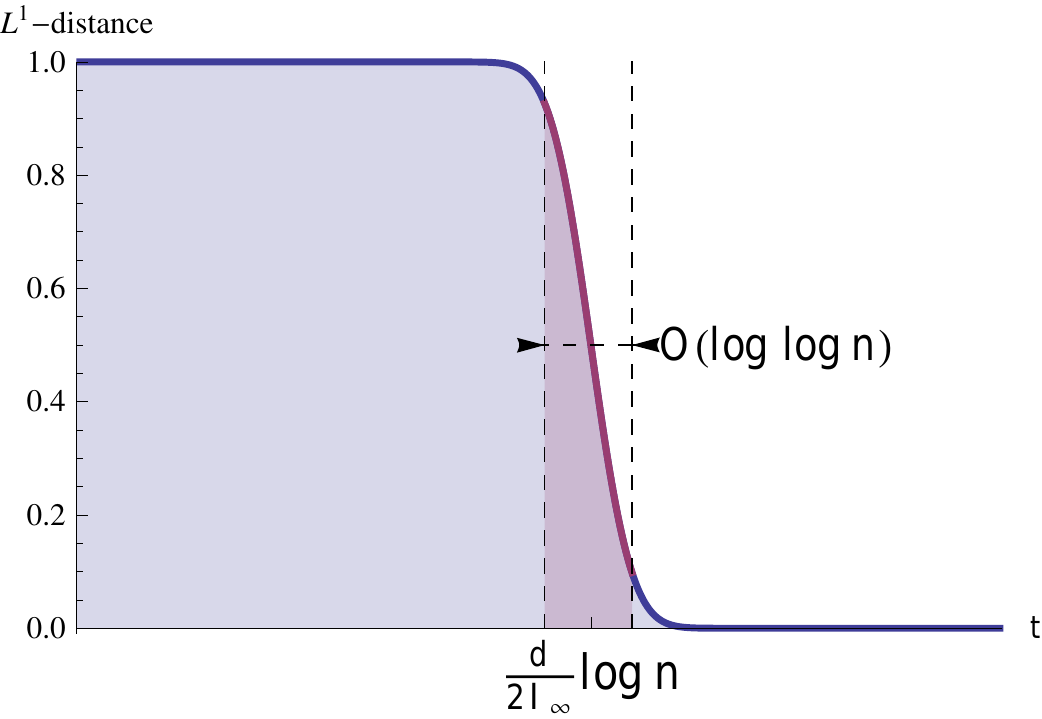}
\caption{Cutoff phenomenon for the $L^1$ (total-variation) distance from stationarity
along time in Glauber dynamics for the Ising model on $\Z_n^d$,
 as established by Theorem~\ref{mainthm-Zd}. Highlighted region denotes the cutoff window of $O(\log\log n)$.}
\label{fig:cutoff}
\end{figure}

In the above theorem, the term \emph{cutoff window} refers to the rate at which the $L^1$-distance from stationarity drops from near $1$ to near $0$.
More precisely, let $\tmix(\epsilon)$ be the minimum $t \geq 0$ where the heat-kernel $H_t$ associated
with a Markov chain is within a total-variation distance of $\epsilon$ from stationarity.
A family of chains is said to exhibit cutoff if for every fixed $0 < \epsilon < \frac12$
we have $\tmix(\epsilon)/\tmix(1-\epsilon) \to 1$ as the system size tends to $\infty$.
A sequence $w_n$ is said to be a cutoff window if $\tmix(\epsilon) = \tmix(1-\epsilon) + O(w_n)$ for every $\epsilon$.

Our results hold for $(\Z/n\Z)^d$ in any dimension $d$, inverse-temperature $\beta$ and
external field $h$ so that the corresponding static Gibbs distribution has a spatial dependence property
known as \emph{strong spatial mixing} (and also as \emph{regular complete analyticity}). This property,
 defined by Martinelli and Olivieri in their seminal paper \cite{MO},
 holds in all regimes where $O(\log n)$ mixing is known for Glauber dynamics for the Ising model. % and it implies a uniformly bounded log-Sobolev constant.
 In particular (see \cites{MO,MO2,MOS}), on $\Z^2$ there is strong spatial mixing
for any $\beta$ with an external field $h \neq 0$,
as well as for any $0 \leq \beta < \beta_c$ when $h=0$.

The next result settles the conjecture of Peres for cutoff for the high temperature regime of the Ising model in any dimension $d \geq 1$.

\begin{maintheorem}\label{mainthm-Zd}
Let $d\geq 1$ and consider the continuous-time Glauber dynamics for the ferromagnetic Ising model on $(\Z/n\Z)^d$ with
periodic boundary conditions, inverse-temperature $\beta$ and external field $h$.
Suppose that $\beta,h$ are such that there is strong spatial mixing. Then the dynamics exhibits cutoff at
$(d/2\lambda_\infty) \log n$ with a window of $O(\log\log n)$, where $\lambda_\infty$ is the spectral gap of the dynamics on the
infinite-volume lattice.
\end{maintheorem}

In the special case of $d=1$ and no external field, it is known that strong spatial mixing always holds and
that the spectral gap at inverse-temperature $\beta$ is $1-\tanh(2\beta)$ independent of the system size (cf., e.g., \cite{LPW}).
We thus have the following corollary to establish the asymptotic mixing time of the one-dimensional Ising model, answering
the aforementioned question of \cite{LPW}.
\begin{maincoro}\label{maincoro-1d}
For any $\beta\geq 0$, the continuous-time Glauber dynamics for the Ising model on $\Z/n\Z$ with periodic boundary conditions, inverse-temperature $\beta$ and no external field has cutoff at $\frac12(1-\tanh(2\beta))^{-1} \log n$.
\end{maincoro}

Our proofs determine the cutoff location in terms of $\lambda(r)$, the spectral gap of the dynamics on the $d$-dimensional lattice $(\Z/r\Z)^d$
for a certain $r=r(n)$. As a biproduct, we are able to relate the spectral gap on tori of varying sizes and obtain that they converge polynomially fast
to the spectral gap of the dynamics on the infinite-volume lattice.
\begin{maintheorem}\label{mainthm-spectral}
For $d\geq 1$ let $\lambda(n)$ be the spectral gap of the continuous-time Glauber dynamics for the Ising model on $(\Z/n\Z)^d$ with
inverse-temperature $\beta \geq 0$ and external field $h$.
If there is strong spatial mixing for $\beta,h$ then
\[\left|\lambda(n) - \lambda_\infty\right| \leq n^{-1/2+o(1)}\,,\]
where $\lambda_\infty$ is the spectral gap of the dynamics on the infinite-volume lattice.
\end{maintheorem}

\subsection{Background and previous work}\label{sec:intro-previous}
While our results hold in greater generality, we will focus on single-site uniform interactions for the sake of the exposition:
The \emph{Ising model} on a finite graph with vertex-set $V$ and edge-set $E$
is defined as follows. Its set of possible configurations is $\Omega=\{\pm1\}^V$, where each configuration
corresponds to an assignment of plus/minus spins to the sites in $V$. The probability that the system is in a
configuration $\sigma \in \Omega$ is given by the Gibbs distribution
\begin{equation}
  \label{eq-Ising}
  \mu(\sigma)  = \frac1{Z(\beta)} \exp\left(\beta \sum_{uv\in E} \sigma(u)\sigma(v) + h \sum_{u \in V} \sigma(u)\right) \,,
\end{equation}

where the partition function $Z(\beta)$ is a normalizing constant.
The parameters $\beta$ and $h$ are the inverse-temperature and external field respectively; for $\beta \geq 0$ we say that the
model is ferromagnetic, otherwise it is anti-ferromagnetic.
These definitions extend to infinite locally finite graphs (see e.g.\ \cites{Liggett,Martinelli97}).

We denote the boundary of a set $\Lambda\subset V$ as the neighboring sites of $\Lambda$ in $V\setminus\Lambda$
and call $\tau\in\{\pm1\}^{\partial \Lambda}$ a \emph{boundary condition}. A periodic boundary condition on $(\Z/n\Z)^d$
corresponds to a $d$-dimensional torus of side-length $n$.

The \emph{Glauber dynamics} for the Ising model is a family of continuous-time Markov chains on the state space $\Omega$,
reversible with respect to the Gibbs distribution, given by the generator
\begin{equation}
  \label{eq-Glauber-gen}
  (\mathcal{L}f)(\sigma)=\sum_{x\in V} c(x,\sigma) \left(f(\sigma^x)-f(\sigma)\right)
\end{equation}
where $\sigma^x$ is the configuration $\sigma$ with the spin at $x$ flipped.
The transition rates $c(x,\sigma)$ are chosen to satisfy finite range interactions, detailed balance, positivity and boundedness and translation invariance (see Section~\ref{sec:prelim}).
Two notable examples for the transition rates are
\begin{enumerate}[(i)]
\item \emph{Metropolis}: $  c(x,\sigma) = \exp\Big(2h\sigma(x)+2\beta\sigma(x)\sum_{y \sim x}\sigma(y)\Big)  \;\wedge\; 1\; $.
\item \emph{Heat-bath}:   $\;c(x,\sigma) = \bigg[1+ \exp\Big(-2h\sigma(x)-2\beta\sigma(x)\sum_{y \sim x}\sigma(y)\Big)\bigg]^{-1}\;$.
\end{enumerate}
These chains have useful graphical interpretations: for instance, heat-bath Glauber dynamics is equivalent to
updating the spins via i.i.d.\ rate-one Poisson clocks, each time resetting a spin according to the conditional distribution given its neighbors.

Ever since its introduction in 1925, the static properties of the Ising model and its Gibbs states,
and more recently the Glauber dynamics for this model, have been the focus of intensive research.
A series of breakthrough papers by Aizenman, Dobrushin, Holley, Shlosman, Stroock et al.\
(cf., e.g., \cites{AH,DobShl,Holley,HoSt1,HoSt2,Liggett,LY,MO,MO2,MOS,SZ1,SZ2,SZ3,Zee1,Zee2}) starting from the
late 1970's has developed the theory of the convergence rate of the Glauber dynamics to stationarity.
It was shown by Aizenman and Holley~\cite{AH} that the spectral gap of the dynamics on the infinite-volume lattice
is uniformly bounded whenever the Dobrushin-Shlosman uniqueness condition holds. Stroock and Zegarli{\'n}ski \cites{Zee1,SZ1,SZ3}
 proved that the logartihmic-Sobolev constant is uniformly bounded provided given the Dobrushin-Shlosman mixing conditions (complete analyticity).
Finally, Martinelli and Olivieri \cites{MO,MO2} obtained this for cubes under the more general condition of strong spatial mixing.
This in particular established $O(\log n)$ mixing throughout the uniqueness regime in two-dimensions.
See the excellent surveys \cites{Martinelli97,Martinelli04} for further details.

To conclude this collection of seminal papers that altogether established $O(\log n)$ mixing throughout the regime of strong spatial mixing,
it remains to pinpoint the asymptotics of the mixing time and determine whether or not there is cutoff in this regime.

The cutoff phenomenon was first identified for random transpositions on the symmetric group in \cite{DiSh},
and for the riffle-shuffle and random walks on the hypercube in \cite{Aldous}. The term ``cutoff''
was coined by Aldous and Diaconis in~\cite{AD}, where cutoff was shown for the top-in-at-random card shuffling process.
See \cites{Diaconis,CS,SaloffCoste2} and the references therein for more on the cutoff phenomenon.
In these examples, and most others where cutoff has been rigorously shown, the stationary distribution has
many symmetries or is essentially one-dimensional (e.g.\ uniform on the symmetric group~\cite{DiSh},
uniform on the hypercube~\cite{Aldous} and one-dimensional birth-and-death chains~\cite{DLP2}).
Even for random walks on random regular graphs (where the stationary distribution is uniform), cutoff was only recently
verified \cite{LS2}.

In the context of spin systems, cutoff was conjectured by Peres (see \cite{LLP}*{Conjecture~1}) for the Glauber dynamics
on any sequence of transitive graphs on $n$ vertices where its mixing time is $O(\log n)$.
More specifically, it was conjectured in \cite{LLP} that cutoff holds for the Ising model on any $d$-dimensional torus $(\Z/n\Z)^d$
in the high temperature regime; see also \cite{LPW}*{Question 8}, where the special case of $d=1$ (Ising model on the cycle) was emphasized.

However, so far the only spin-system where cutoff has been established at some inverse-temperature $\beta > 0$
is Glauber dynamics for the Ising model on the complete graph \cites{LLP,DLP}. There, the magnetization (sum-of-spins)
is in fact a one-dimensional Markov chain whose mixing and cutoff govern
that of the entire dynamics. While this result motivates the conjecture on cutoff for the Ising model on lattices,
its proof fails to provide insight for the latter setting since the complete graph has no geometry to consider.

\subsection{Cutoff for spin systems on the lattice}
Our proof of cutoff for the Ising model on the $d$-dimensional lattice, as stated in Theorems~\ref{mainthm-Z2},\ref{mainthm-Zd},
in fact applies to a broad class of spin systems (essentially any monotone or anti-monotone system), underlying geometries and boundary conditions.
We demonstrate this by establishing cutoff for the anti-ferromagnetic Ising model and the gas hard-core model (see, e.g., \cite{LPW} for definitions).
\begin{maintheorem}\label{mainthm-antiferro}
The cutoff result given in Theorem~\ref{mainthm-Zd} for the ferromagnetic Ising model also holds for the anti-ferromagnetic Ising model.
\end{maintheorem}
\begin{maintheorem}\label{mainthm-hardcore-Zd}
Let $d\geq 1$. The following holds for the Glauber dynamics for the gas hard-core model on $(\Z/n\Z)^d$ with
fugacity $\beta$ and periodic boundary conditions. If $\beta$ is such that there is strong spatial mixing, then the dynamics exhibits cutoff at $(d/2\lambda_\infty) \log n$ with window of $O(\log\log n)$, where $\lambda_\infty$ is the spectral gap of the dynamics on the
infinite volume lattice.
\end{maintheorem}

Our methods also establish cutoff on more general spin-systems such as the Potts model,
provided that the temperature is sufficiently high.  This is discussed further in the companion paper~\cite{LS1}.

Furthermore, our results are also not confined to the underlying geometry of the cubic lattice and in fact
Theorem~\ref{mainthm-Zd} holds for any non-amenable translation invariant lattice (e.g.\ triangular/hexagonal/ladder lattice etc.).

Similarly, the choice of periodic boundary conditions is not a prerequisite for establishing cutoff, though it does enables us to
determine the mixing time in terms of the infinite-volume spectral gap.  For \emph{arbitrary}
boundary conditions we can still establish the existence of cutoff
and in various special cases (e.g.\ under the all-plus boundary conditions)
we can pinpoint its location as stated by the following results of the companion paper \cite{LS1}.
\begin{maintheorem}[\cite{LS1}]\label{mainthm-bc}
Let $d\geq 1$ and consider the Glauber dynamics for the Ising model on $(\Z/n\Z)^d$ with
arbitrary boundary conditions. If there is strong spatial mixing then the dynamics exhibits cutoff.
\end{maintheorem}
\begin{maintheorem}[\cite{LS1}]\label{mainthm-plusBC2d}
Consider Glauber dynamics for the Ising model on $(\Z/n\Z)^2$ with all-plus boundary condition. If there is strong spatial mixing
 then the dynamics exhibits cutoff at $(\lambda_\infty \,\wedge\, 2\lambda_{\mathbbm{H}})^{-1}\log n$,
where $\lambda_\infty,\lambda_{\mathbbm{H}}$ are the spectral gaps of the dynamics on the
infinite-volume lattice and on the halfplane with all-plus boundary condition respectively.
\end{maintheorem}

\begin{figure}
\centering \includegraphics[width=4in]{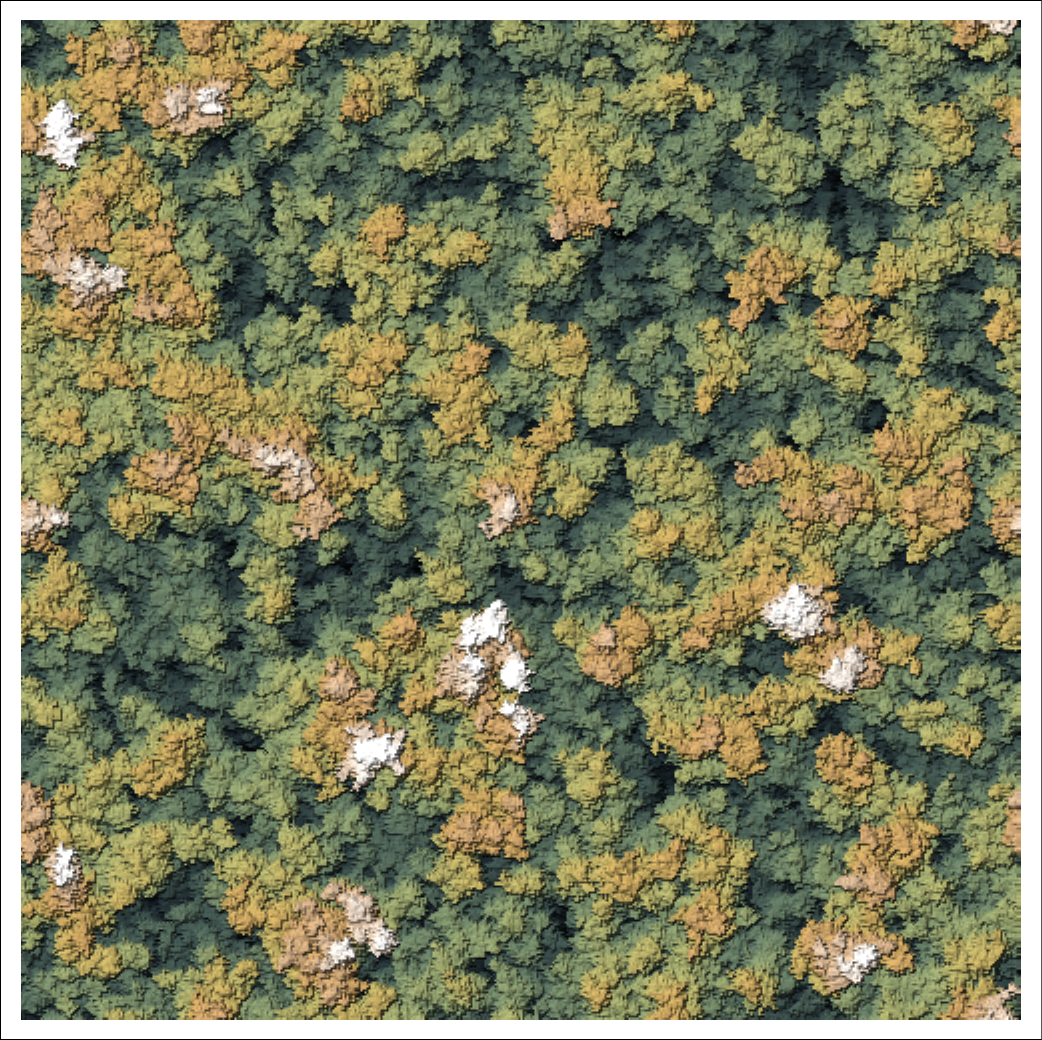}
\caption{Sparse geometry of the update support along a short time interval (mininum subset of sites determining the final configuration, conditioned on the update sequence),
simulated on a $500\times500$ square lattice at $\beta=0.4$.
Color-map highlights the last time a spin belonged to the support (black being earliest and white being latest).}
\label{fig:support}
\end{figure}

\subsection{Main techniques}
To prove total-variation cutoff on the cubic lattice of dimension $d$,
we introduce a technique for bounding the $L^1$-mixing of the dynamics on $(\Z/n\Z)^d$ from below and above in terms of
a quantity that measures the $L^2$-mixing of the dynamics on smaller lattices $(\Z/r\Z)^d$ with $r \asymp \log^3 n$ (see the definition of $\ltwo_t$ in \eqref{eq-mt-def}).
%This appears in Theorem~\ref{thm-l1-l2}.

A key ingredient in this technique is an analysis of the geometry of the ``update support'':
the minimum subset of sites whose initial spins, conditioned on the update sequence in a given time period, determine the final configuration.
We show that even for relatively short update sequences, this minimum subset of sites is typically ``sparse'', i.e.\ comprises remote clusters of small diameter.
Namely, the clusters have diameter $O(\log^3 n)$ and the pairwise distances between distinct clusters are all at least of order $\log^2 n$ (see Definition~\ref{def-sparse-set}
of a sparse set).
Fundamentally, this is ensured by the monotonicity of the system and the strong spatial mixing property, as the dynamics on cubes mixes in timescales much smaller than their diameter.
This is illustrated in Figure~\ref{fig:support} where the clusters are shown in white.

We reduce the analysis of the $L^1$-mixing of the dynamics on $(\Z/n\Z)^d$ to its projection onto sparse supports.
For such subsets, the speed of propagation of information between the distant clusters makes the projections onto them essentially independent.
Using some additional arguments, from here we can translate the above to a product chain on smaller tori $(\Z/r\Z)^d$. Finally, the use of
log-Sobolev inequalities provides a tight control over the $L^1$-mixing in terms of the aforementioned $L^2$-quantity, as stated in Theorem~\ref{thm-l1-l2}.

\section{Preliminaries}\label{sec:prelim}
Consider the Ising model on set of spins $V$, let $\mu$ be its Gibbs distribution as given in \eqref{eq-Ising} and $\sigma \in \Omega=\{\pm1\}^V$ be a configuration.
For $\Lambda\subset V$, we will let $\sigma(\Lambda)$ and $\mu_\Lambda$ denote the projections of $\sigma$
and $\mu$ onto $\{\pm1\}^\Lambda$ respectively.
%The Ising model is a Gibbs measure and satisfies the Markov random field property
%which says that for 3 disjoint subsets $A,B,C\subset V$ if there is no path from $A$ to $B$
% in $V\setminus C$ then $\sigma(A)$ and $\sigma(B)$ are conditionally independent given $\sigma(C)$.
We let $\mu^\tau_\Lambda$ denote the measure on $\Lambda$ given the boundary condition $\tau$, that is,
the conditional measure $\mu_\Lambda(\cdot\mid \sigma_{\partial\Lambda}=\tau)$.
%A probability measure $\mu$ is Gibbs measure for the Ising model if for every finite set $\Lambda\subset V$,
%for any boundary condition $\tau$,  $\mu^\tau_\Lambda$ has the same distribution as the Ising model
%on $\Lambda\cup\partial\Lambda$ with boundary condition $\tau$.
%At least one such measure always exists and the model is said to have \emph{uniqueness} if it is unique.
%On the lattice $\Z^d$  for $d\geq 2$ there is a critical value $\beta_c>0$ such
%that the Ising model has uniqueness when $\beta<\beta_c$ and non-uniqueness when $\beta>\beta_c$.  When $d=1$ uniqueness holds for all $\beta$.
%
%is an Ising model if it satisfies the DLR compatibility conditions that for any finite sets $\Delta\subset\Lambda\subset V$, with boundary conditions $\eta

\subsection{Glauber dynamics for the Ising model}
The Glauber dynamic for the Ising model on the lattice $V = \Z^d$, whose generator is given in \eqref{eq-Glauber-gen},
accepts any choice of transition rates $c(x,\sigma)$ which satisfy the following:
\begin{enumerate}[(1)]
\item \emph{Finite range interactions}: For some fixed $R>0$ and any $x \in V$, if $\sigma,\sigma' \in \Omega$ agree on the ball of diameter $R$ about $x$ then
$c(x,\sigma)=c(x,\sigma')$.
\item \emph{Detailed balance}: For all $\sigma\in \Omega$ and $x \in V$,
\[ \frac{c(x,\sigma)}{c(x,\sigma^x)} = \exp\Big(2h\sigma(x)+2\beta\sigma(x)\sum_{y \sim x}\sigma(y)\Big)\,.
\]
\item \emph{Positivity and boundedness}: The rates $c(x,\sigma)$ are uniformly bounded from below and above by some fixed $C_1,C_2 > 0$.
\item \emph{Translation invariance}: If $\sigma \equiv \sigma'(\cdot + \ell)$, where $\ell \in V$ and addition is according to the lattice metric,
then $c(x,\sigma) = c(x+\ell,\sigma')$ for all $x \in V$.
\end{enumerate}
The Glauber dyanmics generator with such rates defines a unique Markov process, reversible with respect to the Gibbs measure $\mu_V^\tau$.
For simplicity of the exposition, the proof is given for the cases of heat-bath Glauber dynamics or Metropolis Glauber dynamics.
The results for other choices of transition rates follow with minor adjustments to the arguments.

We will let $\P_{\sigma_0}(\cdot)$ and $\E_{\sigma_0}[\cdot]$ denote the probability and expectation conditioned on the Glauber dynamics having initial configuration $\sigma_0$.

\subsection{Mixing, spectral gap and the logarithmic-Sobolev constant}\label{subsec:prelim-gap-sob}
The $L^1$ (total-variation) distance is perhaps the most fundamental notion of convergence in the theory of Markov chains.
For two probability measures $\nu_1,\nu_2$ on a finite space $\Omega$ the total-variation distance is defined as
\[
\|\nu_1-\nu_2\|_\tv = \max_{A\subset \Omega} |\nu_1(A)-\nu_2(A)| = \frac12\sum_{x\in\Omega} |\nu_1(x)-\nu_2(x)| ~,
\]
i.e.\ half the $L^1$-distance between the two measures.
For an ergodic Markov chain $(Y_t)$ with stationary distribution $\nu$, the
mixing-time notion $\tmix$ is defined
with respect to $\|\P(Y_t \in \cdot) - \mu\|_\tv$.

The spectral gap and log-Sobolev constant of the
continuous-time Glauber dynamics are given by the following Dirichlet form
(see, e.g., \cites{Martinelli97,SaloffCoste}):%also \cite{AF}*{Chapter 3,8}):
\begin{align}\label{eq-dirichlet-form}
\gap = \inf_f \frac{\mathcal{E}(f)}
{\var(f)}~,~\quad \sob = \inf_{f} \frac{\mathcal{E}(f)}{\operatorname{Ent}(f)}~,
\end{align}
where the infimum is over all nonconstant $f\in L^2(\mu)$ and
\begin{align*}
%&\mathcal{E}(f) = \left<(I-P)f,f\right>_{\pi} = \frac{1}{2}\sum_{x,y\in\Omega}\left[f(x)-f(y)\right]^2\pi(x)P(x,y)~,\\
\mathcal{E}(f) &= \left<\mathcal{L} f,f\right>_{L^2(\mu)} = \frac12\sum_{\sigma,x} \mu(\sigma)c(x,\sigma)\left[f(\sigma^x)-f(\sigma)\right]^2\,,\\
\operatorname{Ent}(f) &= \E\left[f^2(\sigma) \log \left( f^2(\sigma) /\E f^2(\sigma)\right)\right] ~.
\end{align*}
It is well known (see e.g.\ \cites{DS,AF}) that for any finite ergodic reversible Markov chain $0<2\sob<\gap$ and  $\gap^{-1}\leq \tmix(1/\mathrm{e})$.
In our case, since the sites are updated via rate-one independent Poisson clocks, we also have $\gap\leq 1$.

By bounding the log-Sobolev constant one may obtain remarkably sharp upper bounds not only for
 the total-variation mixing-time but also for the $L^2$-mixing (cf., e.g., \cites{DS1,DS2,DS,DS3,SaloffCoste}). The following theorem
of Diaconis and Saloff-Coste~\cite{DS}*{Theorem 3.7} (see also \cite{SaloffCoste},\cite{AF}*{Chapter 8}) demonstrates this powerful
method.
\begin{theorem}\label{thm-l2-sobolev}
Let $(Y_t)$ be a finite reversible continuous-time Markov chain with stationary distribution $\nu$.
For any $x$ with $\nu(x)\leq \mathrm{e}^{-1}$ and any $s > 0$,
\[ \left\|\P_x (X_s\in \cdot)-\pi\right\|_{L^2(\nu)}\leq \exp\bigg(1-\lambda\left(s-\frac1{4 \sob} \log\log \frac{1}{\nu(x)}\right)\bigg)\,.\]
\end{theorem}

\subsection{Strong spatial mixing and logarithmic-Sobolev inequalities}

As noted in the introduction, bounds on the log-Sobolev constant of the Glauber dynamics for the Ising model were
proved under a variety of increasingly general spatial mixing conditions.
We will work under the assumption of strong spatial mixing (or regular complete analyticity) introduced by a Martinelli and Oliveri~\cite{MO}
as it holds for the largest known range of $\beta$.
\begin{definition}
For a set $\Lambda\subset \Z^d$ we say that $\mathtt{SM}(\Lambda,c_1,c_2)$ holds if there exist constants $c_1,c_2>0$ such that for any $\Delta\subset \Lambda$,
\[
\sup_{\tau,y} \ \Big\| \big(\mu_\Lambda^\tau\big)_\Delta - \big(\mu_\Lambda^{\tau^y}\big)_\Delta  \Big\|_{\tv} \leq c_1 \mathrm{e}^{-c_2
\dist(y,\Delta)}\,,
\]
where the supremum is over all $y \in \partial \Lambda$ and $\tau \in \{\pm1\}^{\partial\Lambda}$ and where $(\mu_\Lambda^\tau)_\Delta$ is the projection of the measure $\mu_\Lambda^\tau$ onto $\Delta$.
We say that strong spatial mixing holds
for the Ising model with inverse temperature $\beta$ and external field $h$ on $\Z^d$ if there exist $L,c_1,c_2>0$ such that $\mathtt{SM}(Q,c_1,c_2)$ holds for all cubes $Q$ of side-length at least $L$.
\end{definition}
The above definition implies uniqueness of the Gibbs measure on the infinite lattice.
Moreover, strong spatial mixing holds for all temperatures when $d=1$ and for $d=2$ it holds whenever $h\neq 0$ or $\beta<\beta_c$.
As discussed in the introduction, this condition further implies a uniform lower bound on the log-Sobolev constant
of the Glauber dynamics on cubes under any boundary condition $\tau$ (see \cites{MO,MO2,Martinelli97}).
We will make use of the next generalization of this result to periodic boundary conditions,
i.e.\ the dynamics on the torus, obtained by following the original arguments as given in \cite{Martinelli97} with minor alterations
(see also \cites{Cesi,GZ,LY,Martinelli04}).

\begin{theorem}\label{thm-log-sobolev-torus}
Suppose that the inverse-temperature $\beta$ and external field $h$ are such that
the Ising model on $\Z^d$ has strong spatial mixing. Then there exists a constant
$\sobinf=\sobinf(\beta,h) > 0$ such that the Glauber dynamics for the Ising model on
$(\Z/n\Z)^d$ with periodic boundary conditions has a log-Sobolev constant  at least $\sobinf$ independent of $n$.
\end{theorem}

\section{Reducing $L^1$ mixing to $L^2$ local mixing}\label{sec:l1-l2}
In this section, we show that the $L^1$ (total-variation) distance of the dynamics on the lattice from stationarity is essentially determined by the $L^2$ distance from stationarity of a projection of this chain onto smaller boxes.

More precisely, consider the continuous-time Glauber dynamics $(X_t)$ for the Ising model on $\Z_n^d$, the $d$-dimensional lattice with side-length $n$ and periodic boundary conditions, and let $\mu$ denote its Gibbs distribution.
Further, consider such a chain on a smaller lattice, namely $(X_t^*)$ on $\Z_r^d$ for $r=3\log^3 n$, and let $\mu^*$ denote its stationary distribution.
(We actually have a lot of freedom in the choice of $r$, e.g.\ any poly-logarithmic value which is at least $\log^{2+\delta} n$ for some fixed $\delta> 0$ would do; this will prove useful later on for
relating the cutoff location to the spectral gap of the dynamics on the infinite-volume lattice.)
 Within this smaller lattice we consider a $d$-dimensional box $B$ with side-length $2\log^3 n$
(the location of the box $B$ within $\Z_r^d$ does not play a role as the boundary is periodic).
Define
\begin{align}
  \label{eq-mt-def}
  \ltwo_t &= \max_{x_0} \left\| \P_{x_0}\big(X^*_t(B) \in \cdot \big)  - \mu^*_{B} \right\|_{L^2(\mu^*_{B})}\,,
\end{align}
where $X^*_t(B)$ and $\mu^*_B$ are the projections of $X^*_t$ and $\mu^*$ resp.\ onto the box $B$.
The following theorem demonstrates how the $L^2$ mixing measured by the quantity $n^d \ltwo^2$ governs the $L^1$ mixing of $(X_t)$.
\begin{theorem}\label{thm-l1-l2}
Let $(X_t)$ be the continuous-time Glauber dynamics for the Ising model on $\Z_n^d$,
and define $\ltwo_t$ as in \eqref{eq-mt-def}.
The following then holds:
\begin{enumerate}[1.]
  \item Let $s=s(n)$ and $t=t(n)$ satisfy $ (10d/\sobinf) \log \log n \leq s < \log^{4/3}n$ and $0<t < \log^{4/3} n$. For any sufficiently large $n$,
\[
   \max_{x_0}\left\| \P_{x_0}(X_{t+s}\in\cdot)-\mu\right\|_\tv \leq
 \frac12 \left(\exp\big((n/\log^5 n)^d \,\ltwo_t^2\big) -1\right)^{1/2} + 3n^{-9d}\,.
\]
In particular, if $(n/\log^5 n)^d \,\ltwo^2_t \to 0$ as $n\to\infty$ for the above $s,t$ then
\[\limsup_{n\to\infty}    \max_{x_0}\left\|\P_{x_0}(X_{t+s}\in\cdot)-\mu\right\|_\tv = 0\,.\]
\item If $(n/ \log^{3} n)^d \ltwo^2_t \to \infty$ for some $t \geq (20d/\sobinf)\log\log n$, then
\[\liminf_{n\to\infty}    \max_{x_0}\left\|\P_{x_0}(X_t\in\cdot)-\mu\right\|_\tv = 1\,.\]
\end{enumerate}
\end{theorem}
\begin{remark}\label{rem-upper-bound-m}
It will be useful to apply Part~1 of the above theorem to lattices of varying sizes.
Indeed, we will show that if $(X_t)$ is the continuous-time Glauber dynamics
for the Ising model on $\Z_m^d$ with
\begin{equation*}
%  \label{eq-m-def}
  \log^3 n \leq m \leq n\,,
\end{equation*}
and with $s,t,\ltwo_t$ as in Theorem~\ref{thm-l1-l2}
(e.g., $s,t < \log^{4/3}n$, the box $B$ in the definition of $\ltwo_t$ given in \eqref{eq-mt-def} has side length $2\log^3 n$, etc.), then
\begin{equation}\label{e:l1-l2varyingm}
    \max_{x_0}\left\| \P_{x_0}(X_{t+s}\in\cdot)-\mu\right\|_\tv \leq
 \frac12 \left(\exp\big(m^d \,\ltwo_t^2\big) -1\right)^{1/2} + 3n^{-9d}\,.\end{equation}
\end{remark}

The upper and lower bounds stated in the above theorem appear in Subsections~\ref{subsec-thm-l1l2-upper-bound}
and \ref{subsec-thm-l1l2-lower-bound} respectively. We begin by describing two key ingredients in the proof of
the $L^1$-$L^2$ reduction that enable us to eliminate long-range dependencies between spins and break down
the lattice into smaller independent blocks.
First we analyze which spins effectively influence the final configuration at some designated
target time and characterize the geometric structure of components comprising such spins.
Second, we introduce the \emph{barrier-dynamics}, a variant of the Glauber dynamics that separates the lattice
into weakly-dependent blocks by surrounding each one with a periodic-boundary barrier.

\subsection{Eliminating long-range dependencies}\label{subsec-long-range}
Consider some time point $t=t(n) \asymp \log n$ just prior to mixing, and let $s=s(n) \asymp \log\log n$ be a short time-frame.
Our goal in this section is to bounds the $L^1$-distance of the Glauber dynamics from equilibrium at time $t+s$ in terms of
the $L^1$-distance projected onto sparse subsets of the spins.
\begin{definition}[\emph{Sparse set}]\label{def-sparse-set}
Let $\log^3 n \leq m \leq n$.
We say that the set $\Delta \subset \Z_m^d$ is \emph{sparse} if for some $L\leq (n /\log^{5} n)^d$ it can be partitioned into components $A_1,\ldots,A_L$ such that
\begin{enumerate}
  [\indent1.] \item Every $A_i$ has diameter at most $\log^3 n$ in $\Z_m^d$.
  \item The distance in $\Z_m^d$ between any distinct $A_i,A_j$ is at least $2d \log^2 n$.
\end{enumerate}
Let $\sparse=\sparse(m)=\{ \Delta\subset\Z_m^d : \mbox{$\Delta$ is sparse}\}$.
\end{definition}
\begin{theorem}\label{thm-xt-bound-a-sparse}
For $\log^3 n \leq m \leq n$ let $(X_t)$ be the Glauber dynamics on $\Z_m^d$ and $\mu$ its stationary measure. Let $(10d/\sobinf) \log\log n \leq s \leq \log^{4/3} n$ and $t>0$. Then there exists some distribution $\nu$ on $\sparse(m)$ such that
\begin{align*}
\left\| \P_{x_0}(X_{t+s}\in\cdot)-\mu\right\|_\tv &\leq \int_{\sparse}\left\|\P_{x_0}(X_t(\Delta)\in\cdot)-\mu_{\Delta} \right\|_\tv\, d\nu(\Delta) +3 n^{-10d}\,.
 \end{align*}
\end{theorem}
\begin{proof}
To prove the above theorem, we introduce the following variant of the Glauber dynamics which breaks down the dynamics into smaller blocks over which we have better control.
\begin{definition}
  [\emph{Barrier-dynamics}]\label{def-barrier}
Let $(X_t)$ be the Glauber dynamics for the Ising model on $\Z_m^d$.
Define the corresponding \emph{barrier-dynamics} as the following coupled Markov chain:
\begin{enumerate}
  \item Partition the lattice into disjoint $d$-dimensional boxes (or blocks), where each side-length is either $\log^2 n$ or $\log^2 n - 1$.
  \item For each such box $B$, let $B^+$ be the $d$-dimensional box centered at $B$ with side-lengths $\log^2 n + 2\log^{3/2} n$, e.g., if $B$ has side-length $\log^2n$
  \[B^+ = \bigcup_{v \in B}\{ u : \|u-v\|_\infty \leq \log^{3/2} n\}\,.\]
  Let $\psi_{B}$ be a graph isomorphism mapping $B^+$ onto some block $C^+$ (and $B$ onto $C \subset C^+$), where the $C^+$ blocks are pairwise disjoint.
  \item Impose a periodic boundary condition on each $C^+$ (to be thought of as a barrier surrounding it), and run the following dynamics:
  As usual, each site $u$ in $\Z_m^d$ receives updates according to a unit rate Poisson clock.
  Updating $u$ at time $t$ via a variable $I\sim U[0,1]$ in the standard dynamics implies updating every $v = \psi_B(u)$ (for some $B$ with $u \in B^+$) via the same update variable $I$.
\end{enumerate}
\end{definition}
The above Markov chain gives rise to the following randomized operator $\mathcal{G}_s$ (for $s > 0$) on $\{\pm1\}^{\Z_m^d}$.
Given an initial configuration $x_0$ for $\Z_m^d$, we translate it to a configuration for the $C^+$ blocks in the obvious manner, and then run the barrier-dynamics
for time $s$. The output of the operator $\mathcal{G}_s$ is obtained by assigning each $u\in\Z_m^d$ the value of $\psi_B(u)$, where $B$ is
the (unique) block that contains $u$. In other words, we pull-back the configuration from the $C$'s onto $\Z_m^d$ (while discarding the spins of the overlaps).

To simplify the notations, we identify the blocks $C,C^+$ with $B,B^+$ whenever there is no danger of confusion.

Note that for the mixing-time analysis, we are only interested in the behavior of the dynamics up to time $O(\log n)$, and the above parameters were chosen accordingly.
Indeed, the next lemma shows that the barrier-dynamics can be coupled to the original one up to time $(\log n)^{4/3}$ except with a negligible error-probability.
%(for other applications, one could readily prolong this coupling by adjusting the side-length of $B^+$).
\begin{lemma}\label{lem-barrier-couple}
Let $t_0 = (\log n)^{4/3}$. The barrier-dynamics and the original Glauber dynamics are coupled up to time $t_0$
except with probability $n^{-10d}$. That is, except with probability $n^{-10d}$, for any $X_0$ we have $X_{s} = \mathcal{G}_s(X_0)$ simultaneously for all $s \leq t_0$.
\end{lemma}
\begin{proof}
Let $(X_t)$ denote the Glauber dynamics on $\Z_m^d$ and let $(\tilde{X}_t)$ denote the barrier-dynamics with corresponding blocks $B_i$ and $B_i^+$.
Apply the aforementioned coupling between the two processes, where the original Glauber dynamics runs as usual, and the barrier-dynamics
uses the same updates for each of its sites.
That is, if site $u$ is updated in the original dynamics via a uniform real $I\sim U[0,1]$, we update it in every $B^+$ that contains it
using the same $I$ at the same time.

Consider some box $B$ and its block $B^+$. Clearly, the barrier-dynamics on $B^+$ is identical to the original dynamics
until it needs to update $\partial B^+$, the boundary of $B^+$ (in which case $(\tilde{X}_t)$ has periodic conditions whereas $(X_t)$ uses external sites in the lattice).

Therefore, a necessary condition to have $X_t(v)\neq \tilde{X}_t(v)$ for some $v\in B$ is the existence of a path of adjacent sites $u_1,\ldots,u_\ell$ connecting $v$ to $\partial B^+$, and a sequence of times $t_1<\ldots<t_\ell\leq t$ such that site $u_i$ was updated at time $t_i$ (note that $\ell \geq \log^{3/2} n$ by definition). Summing over all $(2d)^\ell$ possible paths originating from $v$, and accounting for the probability that the $\ell$ corresponding rate $1$ Poisson clocks fire sequentially before time $t\leq t_0$ (while recalling that $t_0 = o(\ell)$) it then follows that
\begin{align*}
&\sum_{t \leq t_0} \P\left(\cup_i \left\{ X_t(B_i) \neq \tilde{X}_t( B_i)\right\}\right) \leq t_0 n^d \sum_{\ell \geq \log^{3/2}n} (2d)^\ell \, \P(\Po(t_0) \geq \ell) \\
&\quad \leq 2 t_0 n^d \mathrm{e}^{-t_0}\sum_{\ell \geq \log^{3/2}n} \frac{(2dt_0)^\ell }{\ell!} \leq
n^{d - \sqrt{\log n}} < n^{-10d}\,,
\end{align*}
where the inequalities hold for any large $n$ (with room to spare).
\end{proof}
In light of the above lemma, we can focus on the barrier-dynamics for the sake of proving Theorem~\ref{thm-xt-bound-a-sparse}.
Crucially, suitably distant sites evolve independently in this new setting.

The random operator $\mathcal{G}_s$ is determined by the random update sequence $W_s$ (each update is a tuple $(u_j,t_j,I_j)$, where $u_j$ is the site that was updated, $t_j$ is the time of update and $I_j$ is the unit variable determining the update result). In other words, for any such sequence $W_s$ there exists some deterministic function $g_{W_s} : \{\pm1\}^{\Z_m^d} \to \{\pm1\}^{\Z_m^d}$
so that $\mathcal{G}_s(x) = g_{W_s}(x)$ for all $x$.
Further note that $g_{W_s}$ is monotone, by the monotonicity of the Ising model.
We use the abbreviated form $\P(W_s)$ for the probability of encountering the specific update sequence $W_s$ between times $(0,s)$.

\begin{definition}[\emph{Update support}]\label{def-support} Let $W_s$ be an update sequence for the barrier-dynamics between times $(0,s)$.
The \emph{support} of $W_s$ is the minimum subset $\Delta_{W_s} \subset \Z_m^d$ such that $\mathcal{G}_s(x)$ is a function of $x(\Delta_{W_s})$ for any $x$, i.e.,
\[ g_{W_s}(x) = f_{W_s}(x(\Delta_{W_s})) \mbox{ for some $f_{W_s}:\{\pm1\}^{\Delta_{W_s}}\to \{\pm1\}^{\Z_m^d}$ and all $x$.}\]
In other words, $v \notin \Delta_{W_s}$ if and only if for every initial configuration $x$, modifying the spin at $v$ does not
affect the configuration $g_{W_s}(x)$.  This definition uniquely defines the support of $W_s$.
\end{definition}
%Note that the criterion for $v\in \Delta_{W_s}$ uniquely defines the support of $W_s$.
%Further notice that, since $\tX_0$ is the initial state of the barrier-dynamics, all the overlapping sites in $\tX_0$ (i.e., ones that belong to $B^+ \cap B$ for some blocks $B,B'$) are consistent with each other, and hence $\Delta_{W_s} \subset \cup_i B_i$. In particular, we can think of $\Delta_{W_s}$ as a subset of sites of the lattice $\Z_m^d$ (rather than the collection of extended blocks $\cup_i B_i^+$).

Using this notion of the support of updates in the barrier-dynamics, we can now infer the following upper bound on the $L^1$-distance of the original dynamics to stationarity.
\begin{lemma}\label{lem-xt-bound-aw} Let $(X_t)$ be the Glauber dynamics on $\Z_m^d$, and ${W_s}$ be the random update sequence for the barrier-dynamics
along an interval $(0,s)$ for some $s \leq \log^{4/3} n$. For any $x_0$ and $t>0$,
\begin{align*}
\left\| \P_{x_0}(X_{t+s}\in\cdot)-\mu\right\|_\tv &\leq \int\!\!\left\|\P_{x_0}(X_t(\Delta_{W_s})\in\cdot)-\mu_{\Delta_{W_s}} \right\|_\tv d \P({W_s})+ 2 n^{-10d}.
 \end{align*}
%where $\mu_{\Delta_{W_s}}$ is the stationary measure of the projection of $(X_t)$ onto $\Delta_{W_s}$.
\end{lemma}
\begin{proof}
Let $X_t$ be the Glauber dynamics at time $t$ started from $X_0=x_0$,
as usual let $\Omega_m = \{\pm1 \}^{\Z_m^d}$ and let $Y \in \Omega_m$ be distributed according to $\mu$.
Recalling that $g_{W_s}$ denotes the deterministic function associated with an update sequence $W_s$ for the barrier-dynamics in the interval $(0,s)$, for any random configuration $X \in \Omega_m$ we have
 \begin{align*}
\left\|\P(\mathcal{G}_s(X)\in\cdot) - \P(\mathcal{G}_s(Y)\in\cdot)\right\|_\tv
=  \max_{\Lambda\subset \Omega_m} \left[\P(\mathcal{G}_s(X) \in\Lambda) - \P(\mathcal{G}_s(Y)\in\Lambda)\right] \\
=  \max_{\Lambda\subset \Omega_m} \int\left[ \P(g_{W_s}(X)\in\Lambda) -\P(g_{W_s}(Y)\in \Lambda)\right]\,d\,\P({W_s})\,.\end{align*}
Since $g_{W_s}(X) = f_{W_s}(X(\Delta_{W_s}))$ by definition of $\Delta_{W_s}$, the above is at most
\begin{align*}
 &\int  \max_{\Lambda\subset \Omega_m} \left[ \P\big(f_{W_s}(X(\Delta_{W_s}))\in\Lambda\big) -\P\big(f_{W_s}(Y(\Delta_{W_s}))\in \Lambda\big)\right]\,d\,\P({W_s})\\
\leq &\int \left\| \P\big(X(\Delta_{W_s})\in\cdot\big) - \P\big(Y(\Delta_{W_s})\in\cdot\big)\right\|_\tv\,d\,\P({W_s})\\
= &\int \left\| \P\big(X(\Delta_{W_s})\in\cdot\big) - \mu_{\Delta_{W_s}}\right\|_\tv\,d\,\P({W_s})\,,
 \end{align*}
where the inequality in the second line used the fact that when taking a projection of two measures their total-variation can only decrease.

Since $s < \log^{4/3} n$, by Lemma~\ref{lem-barrier-couple} we can couple $X_{t+s}$ and $\mathcal{G}_s(X_t)$ together
except with probability $n^{-10d}$ and hence
\[
\left\|\P(X_{t+s}\in\cdot) - \P(\mathcal{G}_s(X_t)\in\cdot)\right\|_\tv \leq n^{-10d}\, .
\]
Similarly, by Lemma~\ref{lem-barrier-couple} we can couple $\mathcal{G}_s(Y)$ with the Glauber dynamics run from $Y$ for time $s$ (having the stationary distribution $\mu$) and hence
\[
\left\|\P(\mathcal{G}_s(Y)\in\cdot)-\mu\right\|_\tv \leq n^{-10d}\, .
\]
Combining these estimates, it follows that
 \begin{align*}
\left\|\P(X_{t+s}\in\cdot) - \mu\right\|_\tv &\leq\left\|\P(\mathcal{G}_s(X_t)\in\cdot) - \P(\mathcal{G}_s(Y)\in\cdot)\right\|_\tv + 2n^{-10d} \\
& \leq %\P\big(\tX_s(\cup_i B_i)\in\cdot\big)\right\|_\tv &\leq n^{-10d}\,,
\int \left\| \P\big(X_t(\Delta_{W_s})\in\cdot\big) - \mu_{\Delta_{W_s}}\right\|_\tv\,d\,\P({W_s})+ 2n^{-10d}\,,
\end{align*}
as required.
\end{proof}

\begin{figure}
\centering \includegraphics[width=4.5in]{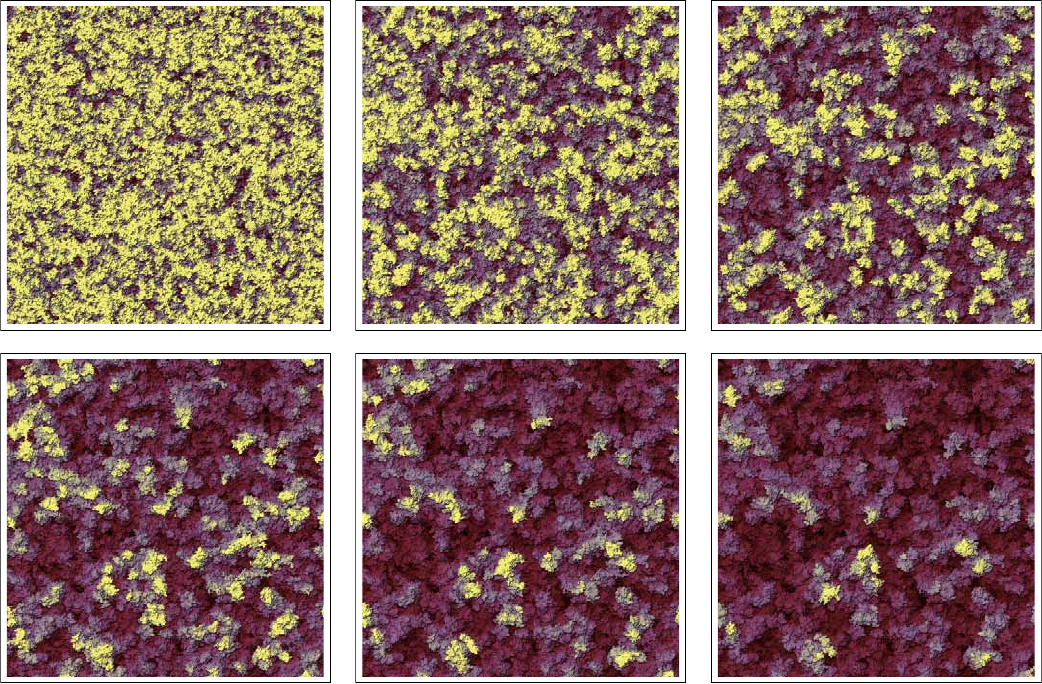}
\caption{Evolution of the update support over increasing time intervals, simulated on a $500\times 500$ square lattice at inverse-temperature $\beta=0.4$.
Highlighted regions correspond to the components comprising the sparse support.}
\label{fig:support-evol}
\end{figure}

Thus far, we have established an upper bound on the $L^1$-distance between $X_{t+s}$ and $\mu$ in terms of $\Delta_{W_s}\subset \Z_m^d$, the support of the update sequence in the barrier-dynamics operator $\mathcal{G}_s$ along the interval $(0,s)$.
We now wish to investigate the geometry of the set of sites comprising $\Delta_{W_s}$ for a typical update sequence ${W_s}$.
The following lemma estimates the probability that $\Delta_{W_s}$ is sparse, as characterized in Definition~\ref{def-sparse-set}.
Figure~\ref{fig:support-evol} shows a realization of the update support becoming sparser with time.
%Let $\sob$ be the infimum over all logarithmic Sobolev constants of the $d$-dimensional tori.
\begin{lemma}\label{lem-sparse-prob}
Let $\mathcal{G}_s$ be the barrier-dynamics operator on $\Z_m^d$, %for $\log^3 n \leq m \leq n$,
let ${W_s}$ be the update sequence up to time $s$ for some $s \geq (10d/\sobinf) \log\log n$, and $\sparse$ be the collection of sparse sets of $\Z_m^d$.
Then $\P(\Delta_{W_s} \in \sparse) \geq 1- n^{-10d}$ for any sufficiently large $n$.
\end{lemma}
\begin{proof}
First, consider a single block $B$ in the barrier-dynamics on $\Z_m^d$ and for simplicity let $\pi$ stand for $\mu_{B^+}$,
the stationary distribution on $B^+$.
Let $E_{B}$ denote the event that $ \Delta_{W_s} \cap B \neq \emptyset$ for a random update sequence ${W_s}$.

Observe that, by definition, the following holds for all $s$: If $\bar{B} \neq B$ are two distinct blocks, then modifying the value of a spin $u \in B$ in the initial configuration $X_0$ can only affect $\tX_{s}(\bar{B}^+)$ provided that $B,\bar{B}$ are adjacent (otherwise the barrier around $\bar{B}$ prevents the effect of this change).
Hence, when assessing whether $u\in \Delta_{W_s}$ it suffices to consider the projection of $\tX_{s}$ onto the block $B$ and all of its neighboring blocks; let $N(B)$ denote this set of $3^d$ blocks.

 Let $(\tX_t^+)$ and $(\tX_t^-)$ be two instances of the barrier-dynamics restricted to $\bar{B}^+$ starting from the all-plus and all-minus
 states respectively, coupled via the monotone coupling (note that the restriction to one block turns these into the standard Glauber dynamics). Clearly, if these two chains coalesce under the update sequence ${W_s}$ at some point $0<t<s$ then either one of them predicts $\tX_{s}(\bar{B}^+)$ regardless of the value of $\tX_0(\bar{B}^+)$, and so
 \[ \P(E_B) \leq \P \bigg( \bigcup_{\bar{B}\in N(B)} X^+_s(\bar{B}^+) \neq X^-_s(\bar{B}^+)\bigg) \leq 3^d \,
\P(\tX_s^+(B^+) \neq \tX_s^-(B^+))
 \,,\]
where the last inequality is by symmetry. As the system is monotone,
\begin{align}
\P(E_B) &\leq 3^d \, \P\left(\tX_s^+(B^+) \neq \tX_s^-(B^+)\right) \leq %\frac{3^d}2 \E\bigg[\sum_{u\in B^+} \tilde{X}^+_s(u)-\tilde{X}^-_s(u)\bigg]\nonumber\\
 3^d \, \sum_{u\in B^+} \P\left(\tX^+_s(u)\neq\tX^-_s(u)\right)   \,.\label{eq-EB-magnetization}
\end{align}
Moreover, for any $u\in B^+$
\begin{align*}
\P\left(\tX^+_s(u)\neq\tX^-_s(u)\right)  &\leq
\| \P(\tX^+_s\in\cdot) - \pi \|_\tv + \| \P(\tX^-_s\in\cdot) - \pi \|_\tv \\
&\leq \tfrac12 \|\P(\tX^+_s\in\cdot) - \pi\|_{L^2(\pi)}
+ \tfrac12\|\P(\tX^-_s\in\cdot) - \pi\|_{L^2(\pi)}\,,
\end{align*}
%$$\P(\tX^+_s(u)=1) - \tfrac12 \leq \| \P(\tX^+_s\in\cdot) - \pi \|_\tv \leq \tfrac12 \|\P(\tX^+_s\in\cdot) - \pi\|_{L^2(\pi)}\,,$$
and by Theorem~\ref{thm-l2-sobolev}, %\cite{AF}*{Chapter 8, Theorem 26},
if the all-plus state $\underline{1}$ has stationary measure at most $\mathrm{e}^{-1}$ (clearly the case for large $n$) then for any $s >0$
\[ \left\|\P (\tX^+_s\in \cdot)-\pi\right\|_{L^2(\pi)}\leq \exp\bigg(1-\gap\left(s-\frac1{4 \sob} \log\log \frac{1}{\pi(\underline{1})}\right)\bigg)\,,\]
where $\lambda$ and $\sob$ are the spectral gap and log-Sobolev constant resp.\ of the Glauber dynamics on $B^+$. Recalling that $\pi(\underline{1}) \geq 1/|\Omega_{B^+}| = 2^{-(1+o(1))\log^{2d} n}$
and that $\gap \geq \sob \geq \sobinf$, the assumption on $s$ gives that
\[ s \geq \frac1{4\sobinf} \log\log (1/\pi(\underline{1})) + \frac{8d}{\gap}\log\log n \]
for any sufficiently large $n$, and in this case
\[ \|\P(\tX^+_s\in\cdot) - \pi\|_{L^2(\pi)} \leq 3 (\log n)^{-8d}\,.\]
By the exact same argument we have
\[ \|\P(\tX^-_s\in\cdot) - \pi\|_{L^2(\pi)} \leq 3 (\log n)^{-8d}\,,\]
and it now follows that
\[ \sum_{u\in B^+} \P\left(\tX^+_s(u)\neq\tX^-_s(u)\right) \Big) \leq (3+o(1)) (\log n)^{-6d} < 4(\log n)^{-6d}\,,\]
where the last inequality holds for large $n$. Combining this with \eqref{eq-EB-magnetization} yields
that for any large $n$, the following holds with room to spare: \[\P(E_{B}) \leq (\tfrac12 \log n)^{-6d}\,.\]

This estimate will now readily imply a bound on the number of components in the support of $W_s$.
Let $E^\sharp$ denote the following event: There exists a collection $\mathcal{B}$ of $L \geq (n /\log^{7}n)^d$ blocks, such that $E_B$
holds (that is, $\Delta_{W_s} \cap B \neq \emptyset$) for all $B\in\mathcal{B}$, and the pairwise distances in blocks between the blocks in $\mathcal{B}$ are all at least $4$.
We claim that $\P(E^\sharp) \leq n^{-20d}$ for large $n$.

To see this, first notice that if a block $B$ has a distance of at least $4$ blocks from a set of blocks $\mathcal{B}$ (i.e., for any $B' \in \mathcal{B'}$, no two blocks in $N(B)$ and $N(B')$ are adjacent), then the variable $\one_{E_B}$ is independent of $\{\one_{E_{B'}} : B\in\mathcal{B}\}$.
Indeed, the initial configuration on $B$ can only affect the outcome of $\tX_s(N(B))$, and by definition this outcome is derived from the initial configuration via the updates of the sites $U_B = \cup_{\bar{B}\in N(B)}\bar{B}^+$. Our assumption on the distance between $B$ and $\mathcal{B}$ precisely implies that the above $U_B$ is disjoint to any $U_{B'}$ for $B' \neq B$ in $\mathcal{B}$, and the statement now follows from the independence of updates to distinct sites.
Hence, as the total number of blocks is $(1+o(1))(n/ \log^{2} n)^d$, for large $n$ we have
\[ \P(E^\sharp) \leq \binom{(2n/\log^{2} n)^d}{(n/\log^{7} n)^d} \left(\tfrac12 \log n\right)^{-6d (n/\log^{7} n)^d}
< n^{-(n /\log^{8}n)^d}\,.\]

Now suppose that there is a sequence of blocks, $(B_{i_0},B_{i_1},B_{i_2},\ldots,B_{i_\ell})$ for some $\ell \geq \ell_0 = \frac1{3d}\log n$, such that
for all $k$,
 \begin{enumerate}
 \item \label{item-prop-1} The distance in blocks between $B_{i_{k-1}},B_{i_{k}}$ is at most $3d$ (i.e., they are the endpoints of a path of at most $3d+1$ adjacent blocks).
 \item \label{item-prop-2} We have $\Delta_{W_s} \cap B_{i_{k}} \neq \emptyset$.
 \end{enumerate}
Clearly, if the distance in blocks between some $B,B'$ is at least $3d+1$, then since every block has side-length at least $\log^2 n - 1$,
the distance between any two sites $u\in B$ and $v\in B'$ is at least $2d\log^2 n $ for sufficiently large $n$.

Observe that if $\Delta_{W_s}\notin \sparse$, then either the event $E^\sharp$ holds or a sequence of blocks as described above must exist.
Indeed, consider some $\Delta_{W_s}$ that is not sparse, and partition it into components as follows:
\begin{align*}
&  u \in B\cap\Delta_{W_s}\mbox{ and } u' \in B'\cap\Delta_{W_s} \mbox{ belong to the same component }\\
&  \Longleftrightarrow\quad \mbox{the distance in blocks between $B,B'$ is at most $3d$.}
\end{align*}
The number of components is clearly at most $(n/\log^7 n)^d$, otherwise the event $E^\sharp$ holds. Furthermore,
by definition (as argued above), the distance between any two distinct components is at least $2d\log^2 n$.
Hence, the assumption that $\Delta_{W_s}\notin\sparse$ implies that some component $A_i$ must have a diameter larger than $\log^3 n$. In particular, there are two sites $u,u'\in A_i$ belonging to $B,B'$ respectively, such that $B,B'$ have distance of at least $\log n$ blocks between them. Moreover, by the way we defined the component $A_i$ there are blocks $B=B_{j_0},\ldots,B_{j_\ell}=B'$ such that $B_{j_i} \cap \Delta_{W_s} \neq \emptyset$ and the distance in blocks between $B_{j_i},B_{j_{i+1}}$ is at most $3d$. As the assumption on $u,u'$ implies that $\ell \geq \frac1{3d}\log n$, this sequence satisfies the required properties.

It therefore remains to show that, except with probability $n^{-20d}$, there does not exist a sequence of blocks satisfying the above properties
\eqref{item-prop-1},\eqref{item-prop-2}.

The above described sequence in particular contains a set $\mathcal{B}$ of at least $\ell/7^d$ blocks, whose pairwise distances are all at least $4$, and $\Delta_{W_s} \cap B \neq \emptyset$ for all $B \in \mathcal{B}$ (for instance, take $\mathcal{B} = \{B_{i_k}\}$, process its blocks sequentially, and for each $B$ delete from $\mathcal{B}$ any $B'\neq B$ whose distance from $B$ is less than 4, a total of
at most $7^d-1$ blocks). By the above discussion on the independence of the events $E_B$ for such blocks, as well as our estimate on $\P(E_B)$, we deduce that for any large $n$
\[ \P(\Delta_{W_s}\cap B \neq \emptyset\mbox{ for all $B\in\mathcal{B}$}) \leq \left(\tfrac12 \log n\right)^{-6d \ell/7^d}\,.\]
Clearly, there are at most $n^d (6d+1)^{d \ell}$ sequences of blocks $\{B_{i_0},B_{i_1},\ldots,B_{i_\ell}\}$, where any two consecutive blocks are of distance (in blocks)
at most $3d$. Hence, the probability that there exists a sequence with the aforementioned properties \eqref{item-prop-1},\eqref{item-prop-2}
is at most
\begin{align*}
n^d \sum_{\ell \geq \ell_0} \left((6d+1)^{7d}  (\tfrac12\log n)^{-6}\right)^{\ell d /7^d} &<  n^d (\log n)^{-(5d/7^d) \ell_0} < n^{d - 7^{-d} \log \log n}
\end{align*}
for any large $n$, as required.
\end{proof}
Theorem~\ref{thm-xt-bound-a-sparse} is now obtained as an immediate corollary of Lemma~\ref{lem-xt-bound-aw} and Lemma~\ref{lem-sparse-prob},
with $\nu$ given by $\nu(\Delta) = \P(\{W_s : \Delta_{W_s} = \Delta\})$ for $\Delta\in\sparse$.  \end{proof}

\subsection{Proof of Theorem~\ref{thm-l1-l2}, Part~1: upper bound on the $L^1$ distance}\label{subsec-thm-l1l2-upper-bound}
Let $\Delta\subset \Z_m^d$ be a sparse set and $\cup_{i=1}^L A_i$ be its partition to components as per Definition~\ref{def-sparse-set}, where $L \leq m^d \,\wedge\,(n/\log^5 n)^d$.
Letting $\dist(\cdot,\cdot)$ denote the distance according to the lattice metric of $\Z_m^d$, put
\[A_i^+ = \{v:\dist(A_i,v) < \log^{3/2} n\}\,,\]
and (recalling that the diameter of $A_i$ is at most $\log^3 n$) let $\psi_i$ be an isometry of $A_i^+$ into a box $B_i$ of side-length $2\log^3 n$ inside $\Z_r^d$, the torus of side-length $r=3\log^3 n$.
A crucial point to notice is that the sets $\{A_i^+\}$ are pairwise disjoint, since the distance between distinct $A_i,A_j$ is at least $2d\log^2 n$.

Let $(X_t^*)$ denote the product chain of the Glauber dynamics run on each of these $L$ copies of $\Z_r^d$ independently, and let $\mu^*$ denote its stationary measure.
Given some initial configuration $x_0$ on $\Z_m^d$, we construct $x_0^*$ for $(X_t^*)$ in the obvious way:
For each $v\in A_i^+$, assign $x_0^*(\psi_i(v)) = x_0(v)$, while the configuration of each $\Z_r^d \setminus \psi_i(A_i^+)$ can be chosen arbitrarily (e.g., all-plus).

Let $t \leq \log^{4/3} n$. We claim that we can now couple $(X_s)$ with $(X_s^*)$ such that $X_s(\Delta) = X_s^*\left(\cup \psi_i(A_i)\right)$ for all $0 \leq s \leq t $ except with probability $n^{-10d}$.

Indeed, by repeating the argument of Lemma~\ref{lem-barrier-couple}, if we let $(X^*_s)$ use the same updates of $(X_s)$ on $\cup A^+_i$ and run independent updates elsewhere, then each $A_i$ can be coupled to $\psi_i(A_i)$ on $[0,t]$ except with probability $n^{-11d}$ (here too the distance between $A_i$ and $\partial A_i^+$ is at least $\log^{3/2} n$). Since the subsets $\{A_i^+\}$ are disjoint, distinct $\psi_i(A_i)$ indeed obtain independent updates in this manner. Summing the above error probability over the $L$ copies we deduce that the coupling exists except with probability $L n^{-11d} \leq n^{-10d}$.

We claim that by the strong spatial mixing property, for any large $n$
\begin{equation}
  \label{eq-mu-mu*-diff}
  \| \mu_{\Delta} - \mu^*_{(\cup \psi_i(A_i))}\|_\tv < n^{-9d}
\end{equation}after one translates
the sites in the obvious manner (that is, when comparing the two measures we apply the isometries $\psi_i$ on the inputs).

It is easy to infer the above statement under the assumption that there is strong spatial mixing $\mathtt{SM}(\Lambda,c_1,c_2)$ for any subset $\Lambda$
(so-called complete analyticity, as opposed to our assumption that strong spatial mixing $\mathtt{SM}(Q,c_1,c_2)$ holds for sufficiently large cubes).
Indeed, recall that for some fixed $c_1,c_2 > 0$ we have
$\|\mu^\tau_\Delta - \mu^{\tau^y}_\Delta\|_\tv \leq c_1 \exp(-c_2 \dist(\Delta,y))$ for all $\tau\in\{\pm1\}^{\partial\Lambda}$ and $y\in\partial\Lambda$.
By taking the boundary conditions $\tau_1$ and $\tau_2$ to be distributed
according to $\mu$ and $\mu^*$ about $\cup \partial A_i^+$ and $\cup\partial \psi_i(A_i^+)$ respectively,
we can transform $\tau_1$ into $\tau_2$ via a series of at most $m^d$ spin flips.
Each flip amounts to an error of at most $c_1 \exp(-c_2\log^{3/2} n)$, leading to \eqref{eq-mu-mu*-diff}.

Since we only have regular complete analyticity at hand, we will obtain \eqref{eq-mu-mu*-diff} via monotonicity and log-Sobolev inequalities.
Take $t = (\log n)^{4/3}$ and apply Theorems~\ref{thm-l2-sobolev} and \ref{thm-log-sobolev-torus} to get %\cite{AF}*{Chapter 8, Theorem 26},
\begin{align*}
\left \| \P_{x_0}(X_t\in\cdot) - \mu \right \|_{\tv}  &\leq \|\P_{x_0} (X_t\in \cdot)-\mu\|_{L^2(\mu)} \\
& \leq \exp\left( 1-\lambda \left(t-\frac1{4 \sob}\log\log \mu^*(x_0)\right)\right)\,,
\end{align*}
where $\lambda,\sob$ are the spectral gap and log-Sobolev constant of $(X_t)$ respectively.
%We let $\lambda(r)$ denote the spectral gap of the Glauber dynamics on the $d$-dimensional torus of side length $r$.
Since $\log\log (1/\mu^*(\sigma)) \geq (d+o(1)) \log n$ for all $\sigma$ and since it holds that $\lambda \geq \sob \geq \sobinf>0$, for a suitably large $n$ we have
\begin{align}\label{e:bigTorusMixSobolev1b}
\max_{x_0} \left \| \P_{x_0}(X_t\in\cdot) - \mu \right \|_{\tv}  &\leq  n^{-10d}\,.
\end{align}
It is well known (see, e.g., \cite{SaloffCoste}*{Lemma 2.2.11} and also \cite{AF}*{Chapter 8}) that the spectral gap (respectively log-Sobolev constant)
of a product chain is equal to the minimum of the spectral gaps (log-Sobolev constants).
In particular, the log-Sobolev constant of $(X_t^*)$ is at least $\sobinf$, hence similarly
\begin{equation}\label{e:toriMixb}
\max_{x_0^*} \left \| \P_{x_0^*}(X^*_t \in\cdot) - \mu^* \right \|_{\tv} \leq n^{-10d}\,.
\end{equation}
Therefore, if $Y$ is a configuration on $\Delta$ distributed according to $\mu_\Delta$ then
\begin{align*}
\big\| \mu_{\Delta} - \mu^*_{(\cup \psi_i(A_i))}\big\|_\tv&\leq \left \| \P(\psi(Y)\in\cdot) - \P_{x_0} (\psi(X_t(\Delta)) \in \cdot) \right \|_{\tv} \\
& +  \left \| \P_{x_0}( \psi(X_t(\Delta)) \in \cdot) - \P_{x_0^*} (X_t^*(\cup \psi_i(A_i)) \in \cdot)\right \|_{\tv} \\
&  + \left \| \P_{x_0^*} (X_t^*(\cup \psi_i(A_i)) \in \cdot) - \mu^*_{(\cup \psi_i(A_i))} \right \|_{\tv} < n^{-9d}\,,
\end{align*}
where the last inequality holds for large $n$, combining \eqref{eq:toriCoupling}, \eqref{e:bigTorusMixSobolev1b} and \eqref{e:toriMixb} with the fact
that projections can only reduce the total variation distance.

Altogether, letting $\Gamma = \cup B_i$ and abbreviating $\P_{x_0^*}(X_t^* \in \cdot)$ by $\pi^*$, we have
\begin{align}
\! \left\| \P_{x_0}\left(X_t(\Delta)\in\cdot\right) - \mu_{\Delta}\right\|_\tv
  &\leq 2n^{-9d}+  \left\| \P_{x_0^*}\left(X_t^*(\cup \psi_i(A_i))\in\cdot\right)-\mu^*_{(\cup \psi_i(A_i))}\right\|_\tv \nonumber\\
&\leq  2n^{-9d}+ \left\| \pi^*_\Gamma -\mu^*_\Gamma\right\|_\tv\,.\label{eq-Xt(A)-bound}
\end{align}
We next seek an upper bound for the last expression, uniformly over the initial configuration $x_0^*$.
\begin{align}
\left\|\pi_\Gamma^* -\mu^*_{\Gamma}\right\|_\tv &=  \frac12
\sum_{b_1,\ldots,b_L} \left| \pi_\Gamma^*(b_1,\ldots,b_L)-\mu^*_\Gamma(b_1,\ldots,b_L)\right|
%&= \frac12 \sum_{b_1,\ldots,b_L}
%\left|\frac{\pi^*(b_1,\ldots,b_L)}{\mu^*_\Gamma(b_1,\ldots,b_L)}-1\right|
%\mu^*_\Gamma(b_1,\ldots,b_L)\\
=  \frac12 \E_{\mu^*_\Gamma}
\left|\frac{\pi_\Gamma^*}{\mu^*_\Gamma}-1\right| \nonumber\\
&\leq  \frac12 \bigg(\E_{\mu^*_\Gamma}
\left|\frac{\pi_\Gamma^*}{\mu^*_\Gamma}-1\right|^2\bigg)^{1/2}
= \frac12 \bigg(\E_{\mu^*_\Gamma}
\left|\frac{\pi_\Gamma^*}{\mu^*_\Gamma}\right|^2-1\bigg)^{1/2}\,,\label{eq-nu-pi-tv}
%&=\max_{x_0} \frac12 \sum_{b_1,\ldots,b_L}
%\left|\prod_{i=1}^L \frac{\P^*(b_i)}{\mu^*(b_i)}-1\right|
%\mu^*(b_1,\ldots,b_L)
\end{align}
where the inequality was by Cauchy-Schwartz, and the last equality
is due to the fact that
$%\begin{align*}
 \E_{\mu^*_\Gamma} \left[\pi_\Gamma^*/ \mu^*_\Gamma\right] = 1
$%\end{align*}
. Now, since $(X_t^*)$ is a product of $L$ independent instances of Glauber dynamics on $\Z_r^d$, we infer that
\begin{align*}
 \E_{\mu^*_\Gamma}
\left|\frac{\pi_\Gamma^*}{\mu^*_\Gamma}\right|^2 &= \E_{\mu^*_\Gamma} \prod_{i=1}^L \bigg|\frac{\pi^*_{B_i}}{\mu^*_{B_i}}\bigg|^2 =
% \prod_{i=1}^L  \E_{\mu^*_{B_i}}\bigg|\frac{\pi^*_{B_i}}{\mu^*_{B_i}}\bigg|^2 =
\prod_{i=1}^L \Big(\left\|\pi^*_{B_i} - \mu^*_{B_i}\right\|_{L^2(\mu^*_{B_i})}^2+1\Big)\,,
\end{align*}
and recalling the definition of $\ltwo_t$ in \eqref{eq-mt-def}, it follows that
$ \E_{\mu^*_\Gamma}
\left|\pi_\Gamma^*/ \mu^*_\Gamma\right|^2 $ is at most $(\ltwo_t^2 +1)^L$.
Plugging this in \eqref{eq-nu-pi-tv},
\begin{align*}
\max_{x^*_0}\left\|\P_{x_0}\left(X_t^*(\Gamma)\in\cdot\right)-\mu^*_{\Gamma}\right\|_\tv &\leq  \frac12 \left(\left(\ltwo_t^2+1\right)^L-1\right)^{1/2} \\
 &\leq
 \frac12 \left(\exp\left(L \,\ltwo_t^2 \right) -1\right)^{1/2}\,.
\end{align*}
Altogether, using \eqref{eq-Xt(A)-bound}, we conclude that for any $\Delta \in \sparse$ and $t < \log^{4/3}n$,
\[
\max_{x_0}\left\|\P_{x_0}\left(X_t(\Delta)\in\cdot\right)-\mu_{\Delta}\right\|_\tv \leq  \frac12 \left(\exp\left(L \,\ltwo_t^2 \right) -1\right)^{1/2} +2n^{-9d}\,.
\]
At this point, Theorem~\ref{thm-xt-bound-a-sparse} implies that for any $s,t$ with $0<t < \log^{4/3}n$
and $ (10d/\sobinf) \log \log n \leq s \leq \log^{4/3}n$ and for every large $n$,
\begin{equation}  \label{eq-dtv-upper-bound}
\left\| \P_{x_0}(X_{t+s}\in\cdot)-\mu\right\|_\tv \leq \frac12 \left(\exp\left(L \,\ltwo_t^2 \right) -1\right)^{1/2} +3n^{-9d}\,.
\end{equation}
Recalling that $L \leq m^d \,\wedge\, (n/\log^5 n)^d$ concludes Part~1 of Theorem~\ref{thm-l1-l2} and also establishes the inequality in Remark~\ref{rem-upper-bound-m}.
\qed

%%% Note: at this point the established upper bound already yields the cutoff result for the 1-dimensional case.

\subsection{Proof of Theorem~\ref{thm-l1-l2}, Part~2: lower bound on the $L^1$ distance}\label{subsec-thm-l1l2-lower-bound}
Let
\begin{align*}
  r = 3\log^3 n \,,&\qquad
  L =  \lfloor n/r \rfloor^d \,,\end{align*}
and let $A_1,\ldots,A_{L} \subset \Z_n^d$ be a collection of $d$-dimensional boxes
of side-length $\frac23 r$ satisfying $\|u-v\|_\infty > r/3$ for any $u\in A_i$ and $v\in A_j$ with $i\neq j$.
Similar to the notation of the previous subsection, for each $i\in\{1,\ldots,L\}$ we define
\[A_i^+ = \{v:\dist(A_i,v) \leq r/6\}\,.\]
Denote the unions of these boxes by $\Delta = \cup_{i=1}^{L} A_i$ and $\Delta^+ = \cup_{i=1}^{L} A_i^+$.

Let $B^+_1,\ldots,B^+_{L}$ be a sequence of disconnected $d$-dimensional boxes of side-length
$r$ and let $\Gamma^+$ denote the graph of their union.
% These are of course much smaller tori than our large torus $G$.
Let $\psi_i$ be an isometry mapping $A_i^+$ to $B^+_i$ and let $\psi$ be the isometry that maps $\Delta^+$ to $\Gamma^+$
such that its restriction to any individual $A^+_i$ is $\psi_i$.
We define $B_i=\psi(A_i)$ and $\Gamma = \cup_{i=1}^{L} B_i$.
For a configuration $X$ on some $\Delta' \subseteq \Delta^+$ we let will denote $\psi(X)$ as the corresponding configuration on $\psi(\Delta')$.

We couple the Glauber dynamics on $\Z_n^d$ and $\Gamma^+$ as follows:
Whenever a site $u\in \Delta^+$ receives an update via some unit variable $I$ we also update the site $\psi(u)$ using the same $I$ and a periodic boundary condition
on its corresponding box $B_i^+$.
Denote the dynamics induced on $\Gamma^+$ as $X^*_t$, and let $\mu^*$ be its stationary distribution.  The above defined coupling satisfies the following claim.
\begin{claim}\label{cl:torusCoupling}
Let $(X_t)$ and $(X^*_t)$ be the above coupled Glauber dynamics on $\Z_n^d$ and $\Gamma^+$ respectively.
Suppose $X_0,X^*_0$ satisfy $\psi(X_0(\Delta^+)) = X_0^*(\Gamma^+)$. Then with probability at least $1-n^{-10d}$,
for all $0\leq s \leq  (\log n)^{4/3}$ we have
\begin{equation}\label{eq:toriCoupling}
\psi(X_s(\Delta))= X_s^*(\Gamma)\,.
\end{equation}
Furthermore,
\[
\max_{x_0} \left \| \P_{x_0}(X_s \in \cdot) - \mu \right \|_{\tv} \geq \max_{x_0^*}  \left \| \P_{x_0^*}(X_s^*(\Gamma) \in \cdot) - \mu^*_{\Gamma} \right \|_{\tv} - 4n^{-10d} \,.
\]
\end{claim}
\begin{proof}
Equation~\eqref{eq:toriCoupling} holds by a simple adaption of the proof of Lemma~\ref{lem-barrier-couple} as the initial conditions and updates agree on
$A_i^+$ and $B_i^+$ for all $i$, and since each $B_i$ is distance $r/6 = \frac12\log^3 n$ from the boundary of $B_i^+$.

For the second statement of the claim, we repeat the argument that yielded inequality~\eqref{eq-mu-mu*-diff} in the previous subsection.
Take $t = (\log n)^{4/3}$ and apply Theorems~\ref{thm-l2-sobolev} and \ref{thm-log-sobolev-torus} to get %\cite{AF}*{Chapter 8, Theorem 26},
\begin{align*}
\left \| \P_{x_0}(X_t\in\cdot) - \mu \right \|_{\tv}  &\leq %\|\P_{x_0} (X_t\in \cdot)-\mu\|_{L^2(\mu)} \\ & \leq
\exp\left( 1-\lambda \left(t-\frac1{4 \sob}\log\log \mu^*(x_0)\right)\right)\,,
\end{align*}
where $\lambda$ and $\sob$ are the spectral gap and log-Sobolev constant of $(X_t)$ respectively. This yields the following for any sufficiently large $n$:
%We let $\lambda(r)$ denote the spectral gap of the Glauber dynamics on the $d$-dimensional torus of side length $r$.
%Since $\log\log (1/\mu^*(\sigma)) \geq (d+o(1)) \log n$ for all $\sigma$ and since $\lambda \geq \sob \geq \sobinf>0$, for a suitably large $n$ we have
\begin{align}\label{e:bigTorusMixSobolev1}
\max_{x_0} \left \| \P_{x_0}(X_t\in\cdot) - \mu \right \|_{\tv}  &\leq  n^{-10d}\,.
\end{align}
Since the log-Sobolev constant of the dynamics on $\Gamma^+$ is at least $\sobinf$,
\begin{equation}\label{e:toriMix}
\max_{x_0^*} \left \| \P_{x_0^*}(X^*_t \in\cdot) - \mu^* \right \|_{\tv} \leq n^{-10d}\,.
\end{equation}
Now, if $Y \in \{\pm1\}^\Delta$ is distributed according to $\mu_\Delta$ then
\begin{align*}
\left \|  \P(\psi(Y) \in \cdot)  - \mu^*_{\Gamma}\right \|_{\tv}&\leq \left \| \P(\psi(Y)\in\cdot) - \P_{x_0} (\psi(X_t(\Delta)) \in \cdot) \right \|_{\tv} \\
& +  \left \| \P_{x_0}( \psi(X_t(\Delta)) \in \cdot) - \P_{x_0^*} (X_t^*(\Gamma) \in \cdot)\right \|_{\tv} \\
&  + \left \| \P_{x_0^*} (X_t^*(\Gamma) \in \cdot) - \mu^*_{\Gamma} \right \|_{\tv} \leq 3n^{-10d}\,,
\end{align*}
where the last inequality added \eqref{eq:toriCoupling}, \eqref{e:bigTorusMixSobolev1} and \eqref{e:toriMix} to the fact
that projections can only reduce the total variation distance. It follows that
\begin{align*}
\max_{x_0} \left \| \P_{x_0}(X_s \in \cdot) - \mu \right\|_{\tv} &\geq \max_{x_0} \left \| \P_{x_0}( \psi(X_s(\Delta))\in\cdot) -  \P(\psi(Y)\in\cdot) \right \|_{\tv} \\
&\geq \max_{x_0^*}  \left\|  \P_{x_0^*}(X_s^*(\Gamma)\in \cdot) - \mu^*_{\Gamma} \right \|_{\tv} - 4n^{-10d}\,,
\end{align*}
where we used the facts that
$\left \|  \P(\psi(Y) \in \cdot)  - \mu^*_{\Gamma}\right \|_{\tv} \leq 3n^{-10d}$ and that $\|  \P(\psi(X_s(\Delta))\in\cdot) - \P_{x_0^*}(X_s^*(\Gamma)\in \cdot)  \|_{\tv} \leq n^{-10d}$.
\end{proof}

%For $t \geq (20d/\sob)\log\log n$ and let $x_0=x_0(t)$ be the initial configuration on $T_i$ which maximizes the quantity
%\[
%\ltwo_t \deq \max_{x_0} \| \P\left(X_t(B_i^*) \in \cdot) \mid X_0^*(T_i)=x_0\right)  - \mu^{*B_i^*} \|_{L^2(\mu^{*B_i^*})}
%\]
%which by symmetry does not depend on $i$.
Recall that the $B_i$ boxes have side-length $\frac23 r = 2\log^3 n$, matching the boxes $B$ in the definition \eqref{eq-mt-def} of $\ltwo_t$.
Let $x^*_0=x^*_0(t)$ be a configuration on a box $B$ of side-length $3\log^3 n$ which achieves $\ltwo_t$, i.e.\
\[
  \ltwo_t = \left\| \P_{x_0^*}\big(X^*_t(B) \in \cdot \big)  - \mu^*_{B} \right\|_{L^2(\mu^*_{B})}\,.
\]
We define i.i.d.\ random variables
\begin{equation}
Y_i = \frac{ \P\left(X_t^*(B_i) =U_i \mid X_0^*(B^+_i)=x_0^*\right) }{ \mu^*_{B_i}(U_i)}\,,
\end{equation}
where the $U_i$ are i.i.d.\ configurations on $B_i$ distributed according to $\mu^*_{B_i}$.  As the dynamics on different tori are independent it follows that the $Y_i$ are independent.
As we will soon show, these random variables provide crucial insight into the mixing of $(X_t)$ in the $L^1$-distance. First we need to obtain some estimates on their moments.  Clearly,
\[
\E Y_i = \sum_{b_i} \frac{ \P\left(X_t^*(B_i) =b_i \mid X_0^*(B^+_i)=x_0^*\right) }{ \mu^*_{B_i}(b_i)}  \mu^*_{B_i}(b_i) = 1\,,
\]
and
\begin{align*}
\var Y_i &= \sum_{b_i} \bigg | \frac{ \P\left(X_t^*(B_i) =b_i \mid X_0^*(B^+_i)=x_0^*\right) }{ \mu^*_{B_i}(b_i)}	  -1 \bigg|^2 \mu^*_{B_i}(b_i)  = \ltwo_t^2 \,.
\end{align*}
Moreover, by a standard $L^\infty$ to $L^2$ reduction (cf., e.g., \cite{SaloffCoste}),
\begin{align*}
\|Y_i-1\|_\infty &=
 \| \P\left(X_t^*(B_i) \in \cdot \mid X_0^*(B_i^+)=x_0^*\right) -  \mu^*_{B_i} \|_{L^\infty(\mu^*_{B_i})} \\
& \leq  \| \P\left(X^*_{t/2}(B_i) \in \cdot \mid X_0^*(B_i^+)=x_0^*\right)  -  \mu^*_{B_i} \|^2_{L^2( \mu^*_{B_i})}\,,
\end{align*}
and hence by Theorems~\ref{thm-l2-sobolev} and \ref{thm-log-sobolev-torus} this is at most $\mathrm{e}^{-c \log\log n}$ for some absolute constant $c>0$, yielding
\begin{align*}
\E |Y_i - 1|^3 \leq  \|Y_i - 1\|_\infty  \var Y_i \leq  \mathrm{e}^{- c\log\log n} \ltwo_t^2 = o(\ltwo_t^2)\,.
\end{align*}
Define $Z_i = \log Y_i$. Taking Taylor series expansions gives
\begin{align*}
  \E Z_i &= \E (Y_i - 1) - \tfrac12 \E (Y_i-1)^2 + O\left( \E |Y_i-1|^3  \right) = -\tfrac{1-o(1)}2 \ltwo_t^2\,,
\end{align*}
and similarly,
\begin{align*}
  \E Z_i^2 &= \E  (Y_i - 1)^2 + O\left( \E |Y_i-1|^3  \right) = (1+o(1)) \ltwo_t^2\,.
\end{align*}

We now derive the required lower bound on the $L^1$-distance of $(X_t)$ from stationarity
at time $t \geq (20d/\sobinf)\log\log n$, provided that
$\lim_{n\to\infty} L \ltwo^2_t = \infty$.
Consider a starting configuration $X_0$ such that $\psi(X_0(A_i^+))=x_0^*$ for all $i$, and
similarly choose $X_0^*$ so that $X_0^*(B_i^+)=x_0^*$ for all $i$.
By Claim~\ref{cl:torusCoupling} it is sufficient to show that under our hypothesis
\[
\left\| \P_{x_0}( X^*_t(\Gamma) \in \cdot) - \mu^*_{\Gamma} \right \|_\tv \rightarrow 1\,.
\]
By the definition $Y_i$'s,
\begin{align*}
 & \left\|\P(  X^*_t(\Gamma) \in \cdot) - \mu^*_{\Gamma}\right \|_\tv \\
  &\qquad = \frac12
\sum_{b_1,\ldots,b_{L}} \left| \P_{x_0}\left(X_t^*(B_1)=b_1,\ldots,X_t^*(B_{L})=b_{L}\right)-\mu^*_{\Gamma}\left(b_1,\ldots,b_{L}\right)\right| \\
& \qquad = \frac12 \sum_{b_1,\ldots,b_L}
\bigg| \prod_{i=1}^{L} \frac{\P_{x_0} (X_t^*(B_i)=b_i)}{\mu_{B_i}^*(b_i)}-1\bigg| \prod_{i=1}^{L} \mu_{B_i}^*(b_i) \\
& \qquad = \frac12 \E \bigg| \prod_{i=1}^{L} Y_i -1\bigg|
 = \E \bigg|  \exp\bigg(\sum_{i=1}^{L} Z_i \bigg) -1\bigg|^{-} \,,
\end{align*}
where $|a|^{-}$ denotes $\max\{-a,0\}$.  The random variables $Z_i$ are independent, $\|Z_i\|_\infty =o(1)$ and $\sum_{i=1}^{L} \var Z_i\rightarrow \infty$,
hence the Central Limit Theorem implies that \[\lim_{n\to\infty} \frac{\sum_{i=1}^{L} \left(Z_i -\E Z_i \right)}
{\sqrt {L\ltwo_t^2}} = \mathcal{N}(0,1)\,.\]
Since $L \E Z_1 / \sqrt{L \ltwo_t^2} = -\frac{1-o(1)}2\sqrt{L \ltwo_t^2 } \to -\infty$, it then follows that $\sum_{i=1}^{L} Z_i$ converges in probability to $-\infty$
and therefore \[ \E \bigg|  \exp\bigg(\sum_{i=1}^{L} Z_i \bigg) -1\bigg|^{-} \rightarrow 1\,,\]
as required.
\qed

\section{Cutoff for the Ising model}\label{sec:mainthms}
In this section we prove the main results, Theorems~\ref{mainthm-Z2} and \ref{mainthm-Zd}. We first describe how the $L^1$-$L^2$ reduction from the previous section
(Theorem~\ref{thm-l1-l2}) establishes the existence of cutoff. A refined analysis of the $L^2$ distance to stationarity
 (captured by the quantity $\ltwo_t$ in the aforementioned theorem) via log-Sobolev inequalities then
allows us to pinpoint the precise location of the cutoff in terms of spectral gaps of the dynamics on tori of prescribed sizes.
Finally, by applying this tool on tori of varying sizes we obtain as a biproduct that, as the tori side-length tends to infinity, these spectral gaps tend to
$\lambda_\infty$, the spectral gap of the dynamics on the infinite-volume lattice. In turn, this yields the asymptotics of the mixing time in terms of $\lambda_\infty$
and establishes Theorem~\ref{mainthm-spectral}.

\subsection{Existence of cutoff}
Theorem~\ref{thm-l1-l2} already establishes cutoff for the Ising model on $(\Z/n\Z)^d$, although not its precise location.
To see this, recall the definition of $\ltwo_t$ given in \eqref{eq-mt-def} and choose $t^*$ as follows
\[ t^* = \inf\bigg\{ t :  \ltwo_{t}^2 \leq \frac{\log^{3d+1}}{n^d} \bigg\}\,.\]
%By this definition, $(n/\log^3 n)^d \ltwo_t^2 \geq \log n$, hence Part~2 of Theorem~\ref{thm-l1-l2} implies that
%$$\left\| \P_{x_0}(X_t\in \cdot) - \mu \right\|_{\tv} \rightarrow 1\,.$$
%On the other hand by Lemma [upper bound lemma] we have that $$\max_{x_0} \| \P_{x_0}(X_{t+s+1}\in\cdot) - \mu \|_{\tv} \leq \frac12\left (\exp( \left( L\ltwo_{t+1}^2\right)-1\right)^{1/2}+o(1)=o(1)\,,$$ where $s = (10d/\sobinf) \log\log n$ since $L\ltwo_{t+1}^2 \leq \log^{1-2d}n=o(1)$.  By \cite{HS} the mixing time of the Glauber dynamics is $\Omega(\log n)$ so we have established cutoff with a window of order $\log\log n$.
As before, the log-Sobolev inequalities of Theorems~\ref{thm-l2-sobolev} and \ref{thm-log-sobolev-torus} imply that $t^*=O(\log n)$. Let $s = (10d/\sobinf) \log\log n$. Since $\ltwo_t$ is a continuous function, we have
$(n/\log^5 n)^d \ltwo_{t^*}^2 = \log^{1-2d} n=o(1)$, and so by Part~1 of Theorem~\ref{thm-l1-l2}
\[
   \max_{x_0}\left\| \P_{x_0}(X_{t^*+s}\in\cdot)-\mu\right\|_\tv \leq
 \frac12 \left(\exp\big(\log^{1-2d} n\big) -1\right)^{1/2} + 6n^{-9d} = o(1)\,.
\]
Next, the results of \cite{HaSi} imply that the $L^1$ mixing time of the Glauber dynamics for the Ising model on $(\Z/n\Z)^d$ has
order at least $\log n$, hence (by the above inequality) $t^*(n)$ is also of order at least $\log n$.
In particular, $t^* \geq (20d/\sob)\log\log n$ for any sufficiently large $n$, and since by definition $(n/\log^3 n)^d \ltwo_{t^*}^2 = \log n$, it follows from Part~2 of Theorem~\ref{thm-l1-l2} that
\[\max_{x_0} \left\| \P_{x_0}(X_{t^*}\in \cdot) - \mu \right\|_{\tv} = 1-o(1)\,.\]
This establishes cutoff at $t^*$ with a window of $O(\log\log n)$. \qed

\subsection{Cutoff location (asymptotics of the mixing time)}
To obtain the asymptotics of the mixing time, it remains to estimate the parameter $t^*$ introduced above,
that is, to understand the threshold $t(n)$ for $n^d \ltwo_t^2$ to tend to infinity faster than some poly-logarithmic function of $n$.

In what follows, let $\lambda(r),\sob(r)$ be the spectral gap and log-Sobolev constant of the Glauber dynamics on a $d$-dimensional torus of side-length $r$. % and $\ltwo_t$ is as given in \eqref{eq-mt-def}.
\begin{lemma}\label{lem-L2-bounds}
Set $c_0 = \frac{12d}{\sobinf}$, let $\frac{20d}{\sobinf\gapinf} \log\log n \leq t \leq \log^{4/3} n$ and $r = 3\log^3 n $.
For $n$ sufficiently large,
\begin{align}\label{eq-L2-bounds}
 \mathrm{e}^{-\lambda(r)t-c_0\log\log n} - n^{-9d} \leq\ltwo_t \leq  \mathrm{e}^{-\lambda(r) t+c_0\log\log n}\,.
\end{align}
\end{lemma}
\begin{proof}
Let $X^*_t$ denote the Glauber dynamics on $\Z_r^d$ with periodic boundary conditions, let $\mu^*_r$ be its stationary distribution and let $\Omega^*_r$ denote its state space.
Since $\log\log (1/\mu^*_r(\sigma)) \geq (3d+o(1))\log \log n$ for all $\sigma \in \Omega^*_r$ and since $\lambda(r) \leq 1$ (vertices are updated at rate 1),
another application of Theorem~\ref{thm-l2-sobolev} %\cite{AF}*{Chapter 8, Theorem 26}
implies that for large $n$
\begin{align}
\ltwo_t &\leq  \max_{\sigma\in\Omega_r^*} \exp\left( 1-\lambda(r)\left(t-\frac1{4 \sob(r)}\log\log \left(1/\mu_r^*(\sigma)\right)\right)\right) \nonumber\\
&\leq  \mathrm{e}^{-\lambda(r) t+\frac{3d+o(1)}{4\sobinf}\log\log n} \leq  \mathrm{e}^{-\lambda(r) t+c_0\log\log n}\,.
\label{e:bigTorusMixSobolev2}
\end{align}
This establishes the upper bound on $\ltwo_t$. Further, as $t \geq (20d/\sobinf\gapinf)\log\log n$
it follows that $r^{d/2}\ltwo_t \leq \log^{-6d} n = o(1)$.

A standard lower bound on the total variation distance in terms of the spectral gap (cf.\ its discrete-time analogue \cite{LPW}*{equation (12.13)}) gives that
\[
\mathrm{e}^{-\lambda(r) t} \leq 2 \left \| \P(X^*_t \in\cdot)- \mu^*\right \|_\tv \mbox{ for all $t>0$}\,.
\]
Set $s = (10d/\sobinf)\log\log n$. Applying Part~1 of Theorem~\ref{thm-l1-l2} to $\Z_r^d$ with these $s,t$ (recalling Remark~\ref{rem-upper-bound-m} and
plugging in $m=r$ in \eqref{e:l1-l2varyingm}) gives
\begin{align}\label{e:l2UpperBound}
  \mathrm{e}^{-\lambda(r) (t+s) } &\leq 2  \left \| \P(X^*_{t+s} \in \cdot) - \mu^*\right \|_\tv  \nonumber \\
  &\leq  \left(\exp\left({r^d\ltwo_t^2}\right) -1\right)^{1/2}  + 6n^{-9d}  \leq 2r^{d/2}\ltwo_t + 6n^{-9d}\,,
\end{align}
where the last inequality used the fact that for $x<1$ we have $\mathrm{e}^{x}-1\leq 2x$ and $r^d \ltwo_t^2 = o(1)$.
Rearranging equation \eqref{e:l2UpperBound} we have that
\begin{align}\label{e:l2UpperBound2}
\ltwo_t \geq   \mathrm{e}^{-\lambda(r)(t+s) -\log(2r^{d/2})}  -   \frac{3n^{-9d}}{r^{d/2}} \geq \mathrm{e}^{-\lambda(r)t - c_0\log\log n}  -   n^{-9d}
\end{align}
Combining equations~\eqref{e:bigTorusMixSobolev2} and~\eqref{e:l2UpperBound} completes the proof.
\end{proof}
The following theorem now establishes the position of the mixing time in terms of $\lambda(r)$ with a window of $O(\log\log n)$.
\begin{theorem}\label{t:cutoffLocation}
Let $(X_t)$ the Glauber dynamics on the $(\Z/n\Z)^d$, and set
\begin{align*}
t^* = t^*(n) &= \frac{d}{2\lambda(r)}\log n\,,\\
t^-_n = t^* - \frac{15d}{\sobinf\gapinf}\log\log n&\,,\qquad  t^+_n = t^* + \frac{25d}{\sobinf\gapinf}\log\log n\,.
\end{align*}
The following then holds: \begin{align*}
\max_{x_0} \left\| \P_{x_0} (X_{t_n^-}\in \cdot) - \mu\right\|_{\tv} \to 1\,,\\
\max_{x_0} \left\| \P_{x_0} (X_{t_n^+}\in\cdot) - \mu\right\|_{\tv}  \to  0\,.
\end{align*}
\end{theorem}

\begin{proof}
We begin by applying Lemma~\ref{lem-L2-bounds} to $t_n^-$. The left-hand-side of \eqref{eq-L2-bounds} establishes that
\begin{align*}
(n/\log^3 n)^d   \ltwo_{t_n^-}^2 &\geq (n/\log^3 n)^d \left(n^{-d} (\log n)^{\frac{6d}{\sobinf}} - n^{-9d}\right) \geq (1+o(1))\log^{3d} n\,,
\end{align*}
and in particular $(n/\log^3 n)^d   \ltwo_{t_n^-}^2\to\infty$, hence by Part~2 of Theorem~\ref{thm-l1-l2}
\[
\max_{x_0} \left\| \P_{x_0} (X_{t_n^-}\in \cdot) - \mu\right\|_{\tv} \to 1\,.
\]
Using the right-hand-side of \eqref{eq-L2-bounds} for $t=t_n^+ - (10d/\sobinf)\log\log n$ gives
\begin{align*}
(n/\log^5 n)^d \ltwo_{t}^2 &\leq (n/\log^5 n)^d \cdot n^{-d} (\log n)^{-\frac{6d}{\sobinf}} \leq \log^{-11d} n = o(1)\,.
\end{align*}
Therefore, by Part~1 of Theorem~\ref{thm-l1-l2}, at time $t_n^+ = t+s$ we have
$$\max_{x_0} \left\| \P_{x_0} (X_{t_n^+}\in\cdot) - \mu\right\|_{\tv}  \to  0\,,$$
as required.
\end{proof}

%\subsection{Limit of spectral gaps}
%
%Know: $T_2(\epsilon \delta) \leq T_s(K,\epsilon) + T_r(K,\delta)$ for $r^{-1}+s^{-1}=3/2$.
%This gives $$\|P^B - \mu^B\|_{t+s}\leq \| P-\mu\|_{t+s}$$

\subsection{Limit of spectral gaps}
Theorem \ref{t:cutoffLocation} established the location of the mixing time in terms of $\lambda(r)$, the spectral gap of the
Glauber dynamics on the $d$-dimensional torus of side-length $r$.
As commented at the beginning of Section~\ref{sec:l1-l2},
we are allowed some latitude in our choice of $r$. This provides a way of relating the spectral gaps for different values of $r$,
ultimately proving that they converge to $\lambda_\infty$, the spectral gap on the infinite-volume lattice.

\begin{lemma}\label{l:spectralGapConvergenceRate}
Let $\beta \geq 0$, and let $\lambda(r)$ be the spectral gap of continuous-time Glauber dynamics for the Ising model on $\Z_r^d$ at inverse-temperature $\beta$.
If there is  strong spatial mixing at inverse temperature $\beta$ then there exists some $\hat{\lambda}>0$ such that
\begin{equation}\label{e:spectralGapConvergenceRate}
\big|\lambda(r) - \hat{\lambda} \big| \leq  r^{-1/2+o(1)} \,.
\end{equation}
\end{lemma}
\begin{proof}
As noted in the beginning of Section~\ref{sec:l1-l2}, there is a lot of freedom in the choice of $r=3\log^3 n$ for the definition of $\ltwo_t$ in \eqref{eq-mt-def},
and the proofs hold as is (while resulting in slightly different absolute constants, e.g., $c_0$ in Theorem~\ref{t:cutoffLocation} etc.) for $r = \log^{2+\delta}n$ with an arbitrary fixed $\delta > 0$.
With this in mind, fix some small $\delta > 0$ and take $r_1 = \log^{2+\delta}n$ and $r_1 \leq r_2 \leq r_1^2$.

In Theorem~\ref{t:cutoffLocation} we provided upper and lower bounds on the quantity $\max_{x_0} \left\| \P_{x_0} (X_{t}\in \cdot) - \mu\right\|_{\tv}$
in terms of $\lambda(r),\sobinf,\gapinf$ and some absolute constants. As argued above, there exists some $C=C(\delta)>0$ such that for every $r_1 \leq r \leq r_1^2$
the statement of the theorem holds with parameters
\begin{align*}
t^* = t^*(n) &= \frac{d}{2\lambda(r)}\log n\,,\\
t^-_n = t^* - \frac{C d}{\sobinf\gapinf}\log\log n&\,,\qquad  t^+_n = t^* + \frac{C d}{\sobinf\gapinf}\log\log n\,.
\end{align*}
Applying this theorem both on $r_1$ and $r_2$ we must have $t_{n}^+(r_1) \geq t_{n}^-(r_2)$ for sufficiently large $n$, and so
\[
 \frac{d}{2\lambda(r_1)}\log n  + \frac{C d}{\sobinf\gapinf} \log\log n \geq  \frac{d}{2\lambda(r_2)}\log n -  \frac{C d}{\sobinf\gapinf}\log\log n\,.
\]
Rearranging we obtain that
\[
\lambda(r_1) - \lambda(r_2) \leq  4 C\frac{\lambda(r_1)\lambda(r_2) }{\sobinf\gapinf} \frac{ \log\log n}{\log n} \leq r_1^{-1/2+\delta} \,,
\]
where the last inequality holds for any sufficiently large $n$.
As we can clearly reverse the role of $r_1$ and $r_2$ it follows that for any large $n$,
\begin{equation}\label{e:spectralGapDifferences}
\max_{r_1< r_2\leq r_1^2} |\lambda(r_1) - \lambda(r_2) | \leq  r_1^{-1/2 + \delta}\,.
\end{equation}
By the above inequality, if $r$ is a sufficiently large integer then
\begin{equation}\label{e:spectralLimit}
\sum_{i=0}^\infty \left|\lambda\big(r^{2^i}\big) - \lambda\big(r^{2^{i+1}}\big)\right| \leq \sum_{i=0}^\infty r^{-2^{i-1}+2^i\delta} \leq 2 r^{-1/2+\delta} < \infty \,.
\end{equation}
Combining equations \eqref{e:spectralGapDifferences} and \eqref{e:spectralLimit} establishes that $\{\lambda(r)\}_{r=0}^\infty$ converges to some limit $\hat{\lambda}$ and that for large $r$,
\[
\left|\lambda(r) - \hat{\lambda} \right| \leq  r^{-1/2+\delta}\,.
\]
Letting $\delta\to 0$ completes the proof.
\end{proof}

It remains to show that the above $\hat{\lambda}$ is equal to $\lambda_\infty$, the spectral gap of the dynamics on the infinite-volume lattice.
Let $(\sigma_t)$ denote the Glauber dynamics on the infinite volume lattice, let $(\sigma^+_t)$ be the dynamics starting from the all-plus configuration and define
\[ \xi_t = \P\left(\sigma_t^+(o)=1\right) - \P\left(\sigma_t^-(o)=1\right)\,,\]
which in the special case of no external field is simply $2\left(\P\left(\sigma_t^+(o)=1\right)-\frac12\right)$.
\begin{claim}\label{cl:magnetizationDecay}
The above defined $\xi_t$ and $\hat{\lambda}$ satisfy
$\big| \xi_t^{1/t} - \exp(-\hat{\lambda})\big| = O\big(\frac{\log t}{t} \big)$,
and in particular $\lim_{t\to\infty} \xi_t^{1/t} = \exp(- \hat{\lambda})
$.
\end{claim}
\begin{proof}
Put $n=\exp(t)$ and let $(X_t^+)$ and $(X_t^-)$ denote the Glauber dynamics starting from the all-plus and all-minus configurations respectively on the torus with side-length $r=3\log^3 n$.
Consider the monotone coupling of $(X_t^+)$ and $(X_t^-)$.
By the symmetry of the torus, the expected number of disagreements at time $t$ is given by $r^d\left(\P(X^+_t(o)=1) -
\P(X^-_t(o)=1)\right) $ and hence
\begin{align}
\P(X^+_t(o)=1) -
\P(X^-_t(o)=1) &\leq  2\max_{x_0} \| \P_{x_0}(X_t\in\cdot)-\mu \|_{\tv}  \nonumber\\
&\leq 2r^d\left (\P(X^+_t(o)=1) -
\P(X^-_t(o)=1)\right) \,,\label{e:disagreementTVBound}
\end{align}
where the first inequality is by definition of the total variation distance.

Next, identify the vertices of $\Z_r^d$ with those in $\{x\in \mathbbm{Z}^d :\|x \|_\infty \leq r/2\}$ and couple $(X_t^+)$ and $(\sigma_t)$
via identical updates to identified vertices. By another simple application of the disagreement percolation argument (as in the proof of Lemma~\ref{lem-barrier-couple}),
\[ \P \left(\sigma_t^+(o) \neq X_t^+(o) \right) \leq n^{-10d} \,. \]
Combining this with \eqref{e:disagreementTVBound},
\[ r^{-d}  \max_{x_0} \| \P_{x_0}(X_t\in\cdot)-\mu \|_{\tv}  - n^{-10d}  \leq \xi_{t} \leq 2\max_{x_0} \| \P_{x_0}(X_t\in\cdot)-\mu \|_{\tv} + n^{-10d}.\]
Applying \eqref{eq-L2-bounds} and \eqref{e:l2UpperBound} for a choice of $s= (10d/\sobinf) \log \log n$ now gives
\begin{align*}
\xi_{t} &\leq O(n^{-9d})+ 2r^{d/2}\mathrm{e}^{-\lambda(r) (t-s)+c_0\log\log n}  \leq  \mathrm{e}^{-\lambda(r) t + c_1 \log\log n}
\end{align*}
and
\begin{align*}
\xi_{t} &\geq \frac1{2r^{d}} \mathrm{e}^{-\lambda(r) t} - n^{-10d} \geq  \mathrm{e}^{-\lambda(r) t - c_1 \log\log n}
\end{align*}
for some constant $c_1 > 0$.  Altogether, as $r=3\log^3 n$ and $t=\log n$,
\[
\bigg|\log\frac{\xi_t^{1/t}}{\exp(-\lambda(3t^3))}\bigg| \leq c_1\frac{\log\log n}t =  c_1\frac{\log t}t\,,
\]
which combined with Lemma \ref{l:spectralGapConvergenceRate} completes the proof.
\end{proof}

A result of Holley~\cite{Holley} shows that $\lim_{t\to\infty} \xi_t^{1/t} =\exp(-\lambda_\infty)$.
We will show how this result is quickly recovered from our proof.
Plugging the test-function $f(\sigma) = \one_{\{\sigma(o)=1\}}-\E \one_{\{\sigma(o)=1\}}$
into the characterization of the spectral gap as the slowest rate of exponential decay of the semigroup gives
\[
\exp(-\lambda_\infty) \geq \lim_{t\to\infty} \xi_t^{1/t}=\exp(-\hat{\lambda})\,.
\]
Now fix $\epsilon>0$ and recall the Dirichlet form \eqref{eq-dirichlet-form}, according to which
\[
\lambda_\infty = \inf_{f\in L^2(\{\pm1\}^{\mathbbm{Z}^d},\mu_\infty)}  \frac{\mathcal{E}_{\mu_\infty}(f,f)}{\var_{\mu_\infty}(f)}\,,
\]
where $\mu_\infty$ is the stationary measure of the infinite-volume Ising model.
For any $f\in L^2(\{\pm1\}^{\mathbbm{Z}^d},\mu_\infty)$
with $\mathcal{E}_{\mu_\infty}(f,f)$ we can find a sequence of functions $f_n\in L^2(\{\pm1\}^{\mathbbm{Z}^d},\mu_\infty)$
each of which depends only on a finite number of spins such that $f_n\to f$ in $L^2(\{\pm1\}^{\mathbbm{Z}^d},\mu_\infty)$ and
$\mathcal{E}_{\mu_\infty}(f_n,f_n)\to\mathcal{E}_{\mu_\infty}(f,f)$ (see e.g. the proof of \cite{Liggett}*{Lemma 4.3}).
So take $g\in L^2(\{\pm1\}^{\mathbbm{Z}^d},\mu_\infty)$  depending
only on a finite number of spins such that
\[
\frac{\mathcal{E}_{\mu_\infty}(g,g)}{\var_{\mu_\infty}(g)} \leq \frac{\mathcal{E}_{\mu_\infty}(f,f)}{\var_{\mu_\infty}(f)} +\epsilon.
\]
For some large enough $M$ we have that $g$ is a function of the spins in the box $B_M=\{x\in \mathbbm{Z}^d: \|x\|_\infty\leq M\}$.
We compare the Glauber dynamics on $\Z^d$ to that on $\Z_r^d$ for some large $r$, identifying the vertices of the latter with those in $\{x\in \mathbbm{Z}^d: \|x\|_\infty\leq r/2\}$
and denoting its stationary distribution by $\mu_r$.
By the strong spatial mixing property, the projection of $\mu_{r}$ on $B_M$ converges to the projection of $\mu_\infty$ on $B_M$ as $r\to\infty$
(in fact, an assumption weaker than strong spatial mixing would already infer this, e.g.\ uniqueness),
hence
\[
\frac{\mathcal{E}_{\mu_\infty}(g,g)}{\var_{\mu_\infty}(g)} = \lim_{r\to\infty} \frac{\mathcal{E}_{\mu_r}(g,g)}{\var_{\mu_r}(g)} \geq \lim_{r\to\infty} \lambda(r) =  \hat{\lambda} \,,
\]
where the inequality follows from the characterization of the spectral gap by the Dirichlet form.  This implies that
\[
\lambda_\infty = \inf_f \frac{\mathcal{E}_{\mu_\infty}(f,f)}{\var_{\mu_\infty}(f)} \geq \hat{\lambda} -\epsilon\,,
\]
and letting $\epsilon\to0$ gives that $\lambda_\infty=\hat{\lambda}$, as required.
This completes the proof of Theorems~\ref{mainthm-Z2} and \ref{mainthm-Zd}.
\qed

\section{Cutoff for other spin-systems}\label{sec:othermodels}

While the proof above was given for the ferromagnetic Ising model,
its arguments naturally extend to other well-studied spin-system models.
In this section we outline the minor modifications one needs to make in order to obtain cutoff
for general monotone and anti-monotone systems, thereby proving Theorems~\ref{mainthm-antiferro} and~\ref{mainthm-hardcore-Zd}.

A key prerequisite for our proofs is a uniformly bounded log-Sobolev constant for the dynamics.
This was proved in great generality in~\cite{Martinelli97} for Glauber dynamics on spin-system models on the lattice
where strong spatial mixing holds, and that result carries to periodic boundary with minor adjustments.

Otherwise, the only specific property of the Ising model used in our arguments is the monotonicity of the model.
The Glauber dynamics \eqref{eq-Glauber-gen} for a spin-system is \emph{monotone} (or \emph{attractive}) if
there is a partial ordering of the state space $\preceq$ according to which the transition rates satisfy
\begin{equation}\label{eq-attractive}
  \sigma(x) c(x,\sigma) \leq \eta(x) c(x,\eta) \mbox{ for all $\eta \preceq \sigma$ with $\eta(x)=\sigma(x)$}\,.
\end{equation}
Under this condition, the system allows a monotone coupling of the Glauber dynamics.
We use the monotonicity of the model in precisely two locations.
%In Lemma~\ref{lem-stationary-barrier} we bound the probability that the dynamics started from the all $+$ and all $-$ configurations couples with probability at least $1-n^{-11d}$ which we use to establish the distance between the stationary distributions of the Glauber dynamics and the barrier dynamics.
\begin{enumerate}[(1)]
  \item In Lemma~\ref{lem-sparse-prob} we prove that the update support is sparse w.h.p.\ by showing that in most of the tori
  that make up the barrier dynamics the all-plus and all-minus configurations couple (hence the final configuration is
  independent of the projection of the starting configuration onto those tori). This is a crucial application of the system's monotonicity.
  \item In equation~\eqref{e:disagreementTVBound} of Claim~\ref{cl:magnetizationDecay} we relate the expected number of disagreements
  in the monotone coupling of the dynamics to the total-variation distance of the dynamics from stationarity.
\end{enumerate}
The arguments in both of these steps hold essentially unchanged for any monotone dynamics.

An anti-monotone system is one where the reverse inequality in \eqref{eq-attractive} always holds.
When the underlying geometry is a bipartite graph, there is a standard transformation of an anti-monotone system into a monotone one.
Let $V = V_e\cup V_o$ be a partition of the sites such that there are no edges within $V_e$ or $V_o$.
We define a new partial ordering $\preceq^*$ on $\{\pm1\}^V$ as follows: For two configurations $\sigma,\eta$
we have $\sigma \preceq^* \eta$ if $\sigma(v)\leq\eta(v)$ for all $v\in V_e$ and $\sigma(v)\geq\eta(v)$ for all $v\in V_o$.
It is easy to verify that an anti-monotone system under the standard partial ordering $\preceq$ is a monotone system under $\preceq^*$.

To derive Theorems~\ref{mainthm-antiferro} and \ref{mainthm-hardcore-Zd}, note that the $d$-dimensional lattice with periodic boundary conditions $(\Z/r\Z)^d$ is bipartite
if and only if $r$ is even. Crucially, our proof for cutoff on $(\Z/n\Z)^d$ only required the dynamics to be monotone on the smaller tori $(\Z/r\Z)^d$
for $r=r(n)$ as defined above \eqref{eq-mt-def}. As we noted there, we have freedom for our choice of $r$ (basically any choice between $\log^{2+\epsilon} n$
and $\log^{O(1)}n$ would do) and in particular we can let $r$ be even. Applying the above transformation therefore establishes cutoff for the anti-ferromagnetic
Ising model and the gas hard-core model.

As a side note, in our companion paper~\cite{LS1} we establish that any model with soft interactions (e.g.\ the Potts model) has cutoff at high enough temperatures.
One of the key challenges there is to show that the update support is typically sparse.
We further address non-periodic boundary conditions, which break down the symmetric structure of the torus, thereby letting the smaller
tori (which we analyze in our $L^1$-$L^2$ reduction) play different roles in the dynamics according to their vicinity to the boundary.

\section*{Acknowledgments}
We are grateful to Yuval Peres for inspiring us to pursue this research.
We thank him and Fabio Martinelli for useful discussions.

This work was initiated while the second author was an intern at the Theory Group of Microsoft Research as a doctoral student at UC Berkeley.

\end{document}